\documentclass[a4paper,12pt]{article}

\usepackage{amscd,amssymb,amsmath,amsfonts,amsthm, mathrsfs}
\usepackage{graphicx}
\usepackage{verbatim,enumerate,paralist}
\usepackage{tensor}
\usepackage{textcomp}
\usepackage[colorlinks]{hyperref}
\hypersetup{%
  urlcolor=blue,linkcolor=blue,citecolor=blue
}
\usepackage{pdfpages}
\usepackage[T1]{fontenc}

\usepackage[arrow, matrix, tips, curve, graph, rotate]{xy}
\SelectTips{cm}{10}
\newcommand{\cd}[2][]{\vcenter{\hbox{\xymatrix#1{#2}}}}

\makeatletter
\def\matrixobject@{%
  \edef \next@{={\DirectionfromtheDirection@ }}%
  \expandafter \toks@ \next@ \plainxy@
  \let\xy@@ix@=\xyq@@toksix@
  \xyFN@ \OBJECT@}
\let\xy@entry@@norm=\entry@@norm
\def\entry@@norm@patched{%
  \let\object@=\matrixobject@
  \xy@entry@@norm }
\AtBeginDocument{\let\entry@@norm\entry@@norm@patched}
\makeatother

\newcommand{\twocong}[2][0.5]{\ar@{}[#2] \save ?(#1)*{\cong}\restore}
\newcommand{\twoeq}[2][0.5]{\ar@{}[#2] \save ?(#1)*{=}\restore}
\newcommand{\ltwocell}[3][0.5]{\ar@{}[#2] \ar@{=>}?(#1)+/r 0.2cm/;?(#1)+/l 0.2cm/^{#3}}
\newcommand{\rtwocell}[3][0.5]{\ar@{}[#2] \ar@{=>}?(#1)+/l 0.2cm/;?(#1)+/r 0.2cm/^{#3}}
\newcommand{\utwocell}[3][0.5]{\ar@{}[#2] \ar@{=>}?(#1)+/d  0.2cm/;?(#1)+/u 0.2cm/_{#3}}
\newcommand{\dtwocell}[3][0.5]{\ar@{}[#2] \ar@{=>}?(#1)+/u  0.2cm/;?(#1)+/d 0.2cm/^{#3}}
\newcommand{\ultwocell}[3][0.5]{\ar@{}[#2] \ar@{=>}?(#1)+/dr  0.2cm/;?(#1)+/ul 0.2cm/^{#3}}
\newcommand{\urtwocell}[3][0.5]{\ar@{}[#2] \ar@{=>}?(#1)+/dl  0.2cm/;?(#1)+/ur 0.2cm/^{#3}}
\newcommand{\dltwocell}[3][0.5]{\ar@{}[#2] \ar@{=>}?(#1)+/ur  0.2cm/;?(#1)+/dl 0.2cm/^{#3}}
\newcommand{\drtwocell}[3][0.5]{\ar@{}[#2] \ar@{=>}?(#1)+/ul  0.2cm/;?(#1)+/dr 0.2cm/^{#3}}

\usepackage[usenames,dvipsnames]{pstricks}
\usepackage{epsfig}
\usepackage{pst-grad} 
\usepackage{pst-plot} 

\newcounter{Definitioncount}
\newtheorem{theorem}{Theorem}
\newtheorem{lemma}[theorem]{Lemma}
\newtheorem{proposition}[theorem]{Proposition}
\newtheorem{corollary}[theorem]{Corollary}

\newtheorem{definition}[theorem]{Definition}
\newtheorem{remark}[theorem]{Remark}
\newtheorem{example}[theorem]{Example}

\newdir{ >}{\xybox{*{}+/2.5mm/*@{>}}} 

\newtheoremstyle{fact}{\bigskipamount}{\medskipamount}{\upshape}{}{\itshape}{. }{ }{Fact}
\theoremstyle{fact}

\newtheoremstyle{genquest}{\bigskipamount}{\medskipamount}{\upshape}{}{\itshape}{. }{ }{General Question}
\theoremstyle{genquest}

\newtheoremstyle{step}{2\bigskipamount}{\medskipamount}{\upshape}{}{\itshape}{. }{ }{\underline{Step~\thestep}}
\theoremstyle{step}

\renewcommand{\thestep}{\arabic{step}}

\newcommand{\lra}{\longrightarrow}

\newcommand{\Lra}{\Longrightarrow}

\newcommand{\ldual}[1]{\mathord{{\let\nolimits\relax\sideset{^\wedge}{}{#1}}}}
\newcommand{\laction}[2]{\mathord{{\let\nolimits\relax\sideset{^{#1}}{}{#2}}}}
\newcommand{\conj}[2]{\mathord{{\let\nolimits\relax\sideset{^{#1}}{}{#2}}}}

\newcommand{\dd}{\colon}

\def\CA{{\mathscr A}}
\def\CB{{\mathscr B}}

\def\CK{{\mathscr K}}

\def\CV{{\mathscr V}}
\def\CW{{\mathscr W}}
\def\CX{{\mathscr X}}

\def\KK{{\mathfrak K}}

\begin{document}

\author{Ross Street\footnote{The author gratefully acknowledges the support of Australian Research Council Discovery Grant DP130101969.} \\
{\small{Centre of Australian Category Theory, Macquarie University, NSW 2109, Australia}} \\
{\small{<ross.street@mq.edu.au>}}}

\title{Wreaths, mixed wreaths and twisted coactions}
\date{{\small \today}}
\maketitle

\noindent {\small{\emph{2010 Mathematics Subject Classification:} 18D10; 05A15; 18A32; 18D05; 20H30; 16T30}

\begin{abstract}
\noindent Distributive laws between two monads in a 2-category $\CK$, as defined by Jon Beck in the case 
$\CK=\mathrm{Cat}$, 
were pointed out by the author to be monads in a 2-category $\mathrm{Mnd}\CK$ of monads. 
Steve Lack and the author defined wreaths to be monads in a 2-category 
$\mathrm{EM}\CK$ of monads with different 2-cells from $\mathrm{Mnd}\CK$.

Mixed distributive laws were also considered by Jon Beck, Mike Barr and, later, various others; 
they are comonads in $\mathrm{Mnd}\CK$.  
Actually, as pointed out by John Power and Hiroshi Watanabe, there are a number of dual 
possibilities for mixed distributive laws. 

It is natural then to consider mixed wreaths as we do in this article; 
they are comonads in $\mathrm{EM}\CK$.
There are also mixed opwreaths: comonoids in the Kleisli construction completion 
$\mathrm{Kl}\CK$ of $\CK$. 
The main example studied here arises from a twisted coaction of a bimonoid on a monoid.
Corresponding to the wreath product on the mixed side is wreath convolution,
which is composition in a Kleisli-like construction. 
Walter Moreira's Heisenberg product of linear endomorphisms on a Hopf algebra, is an example of such convolution, actually involving merely a mixed distributive law. 
Monoidality of the Kleisli-like construction is also discussed.
     
\end{abstract}

\tableofcontents

\section{Introduction}\label{Intro}
\noindent While trying to expose the categorical mechanism
behind the Heisenberg product of endomorphisms, as defined and studied in \cite{AFM2015, MoreiraPhD}, we noticed that it has to do with distributive laws 
in the sense of Beck \cite{Beck1969}; 
also see \cite{3, 45} for the general setting and monad terminology. 
A distributive law $\xi\dd TS\Rightarrow ST$ of a monad $T$ 
over a monad $S$ on a category $\CA$ 
gives rise to a monad structure on the
composite $ST$.
For a mixed distributive law $\zeta\dd SG\Rightarrow GS$ of a comonad $G$
over a monad $S$, we do not expect $GS$ to be a monad or comonad, so what takes its place?
It is the (internalized) $\zeta$-parametrized convolution of 2-cells $G\Rightarrow S$. 
To understand this to some extent (externally), consider the 
case where $T\dashv G$, and $\xi$ and $\zeta$ are 
mates \cite{7} under that adjunction. 
The adjunction gives an isomorphism
\begin{eqnarray*}
[\CA,\CA](1_{\CA},ST)\cong [\CA,\CA](G,S) 
\end{eqnarray*}
The monoid structure on the left-hand side, arising
pointwise from the monad structure on $ST$ determined by $\xi$,
transports to a convolution-like monoid structure on the right-hand side,
expressible in terms of $\zeta$.

Rather than remain at the level of distributive law, since there are
articles \cite{HoPa2001, PoWa2002, HHM2007} which study that, 
we decided to generalize to the wreaths of \cite{71}. 

This article begins with a review of wreaths as defined in \cite{71}.
We spend some time extending Example 3.2 of \cite{71} to a 
wreath between monoids rather than groups: the use of fibrations
is to bring out the cohomological aspects which permeate the paper.

As for mixed distributive laws \cite{PoWa2002}, there are several 
possibilities for mixed wreaths. We look at those which are comonads
in either the (limit) completion of the ambient 2-category under Eilenberg-Moore
construction or the (colimit) completion under the Kleisli construction.
The first are called mixed wreaths, the second mixed opwreaths.
Mixed Eilenberg-Moore and mixed Kleisli constructions are described
and their universal properties presented. Composition in a mixed
Kleisli category is convolution parametrized by the mixed wreath. 

Section~\ref{sectiontwcoact} provides the construction of a mixed opwreath
is a dual of the wreath construction appearing as Example 3.3 in \cite{71} based on 
Sweedler's crossed product of Hopf algebras.
We also generalize to bialgebras (bimonoids).
The ingredient is a twisted coaction of a bimonoid on a monoid.
Natural connections to cohomological structures are pursued.

Section~\ref{monoidality} sets out when a mixed opwreath is opmonoidal.
This is about a monoidal structure on the mixed Kleisli construction. 
The final section gives structure on a twisted coaction so that the
associated mixed opwreath becomes opmonoidal.

We will use the string diagrams for monoidal categories as
explained in \cite{37}. However, we read the diagrams from top
to bottom rather than the reverse. 
For example, if $A$ is a monoid in any monoidal category $\CV$, 
the multiplication $\mu = \mu_A$ and unit $\eta = \eta_A$ are respectively depicted as follows.
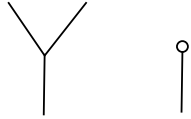
\begin{center}
\psscalebox{0.6 0.6} 
{
\begin{pspicture}(0,-1.2563057)(3.9765608,1.2563057)
\psline[linecolor=black, linewidth=0.04](0.016560897,1.244001)(0.8165609,0.064001024)(1.736561,1.244001)
\psline[linecolor=black, linewidth=0.04](0.8165609,0.10400102)(0.7965609,-1.255999)
\pscircle[linecolor=black, linewidth=0.04, dimen=outer](3.836561,0.264001){0.14}
\psline[linecolor=black, linewidth=0.04](3.836561,0.16400103)(3.816561,-1.195999)
\end{pspicture}
}
\end{center}
If we are dealing with
a braided monoidal category, the braiding $c_{X,Y}\dd X\otimes Y\to Y\otimes X$ will be depicted as a crossing as follows.
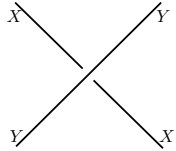
\begin{center}
\psscalebox{0.6 0.6} 
{
\begin{pspicture}(0,-1.6242286)(3.65,1.6242286)
\psline[linecolor=black, linewidth=0.04](3.38,1.6099999)(0.2,-1.57)
\psline[linecolor=black, linewidth=0.04](0.18,1.6099999)(1.66,0.14999995)
\psline[linecolor=black, linewidth=0.04](1.9,-0.11000006)(3.42,-1.61)
\rput[bl](0.0,1.17){$X$}
\rput[bl](3.28,1.17){$Y$}
\rput[bl](3.34,-1.47){$X$}
\rput[bl](0.06,-1.47){$Y$}
\end{pspicture}
}\end{center}

\section{Review of wreaths}\label{row}
 
The free completion $\mathrm{EM}(\mathfrak{K})$ of a 2-category $\mathfrak{K}$
(such as $\mathrm{Cat}$) under the Eilenberg-Moore construction
was identified in \cite{71}.
The objects of $\mathrm{EM}(\mathfrak{K})$ are monads $(\CA,T)$ in $\mathfrak{K}$.
That is, $T$ is a monoid in the endomorphism category $\mathfrak{K}(\CA,\CA)$, 
monoidal under composition as tensor product, so we can draw planar diagrams.
A morphism $(F,\phi)\dd (\CA,T) \to (\CB,S)$ consists of a morphism
$F\dd \CA\to \CB$ and a 2-cell $\phi \dd SF\Rightarrow FT$ in $\mathfrak{K}$
compatible with the monad structures on $T$ and $S$.
A 2-cell $\rho \dd (F,\phi)\Rightarrow (G,\psi)$ is a 2-cell $\rho \dd F\Rightarrow GS$
such that the following equation holds.
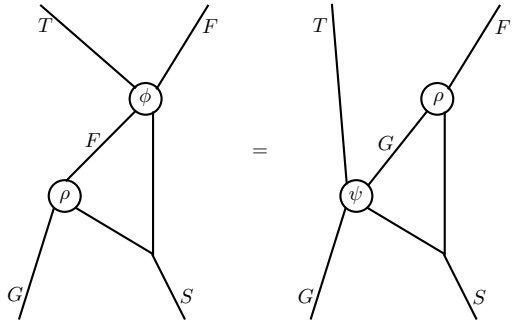
\begin{center}
\psscalebox{0.7 0.7} 
{
\begin{pspicture}(0,-2.9952807)(9.43,2.9952807)
\pscircle[linecolor=black, linewidth=0.04, dimen=outer](2.6,1.1936566){0.32}
\rput[bl](2.48,1.0536566){$\phi$}
\pscircle[linecolor=black, linewidth=0.04, dimen=outer](1.08,-0.6463434){0.32}
\rput[bl](0.98,-0.7863434){$\rho$}
\psline[linecolor=black, linewidth=0.04](0.62,2.9736567)(2.42,1.3936566)
\psline[linecolor=black, linewidth=0.04](3.78,2.9736567)(2.8,1.3736566)
\psline[linecolor=black, linewidth=0.04](2.42,0.9736566)(1.1,-0.3663434)
\psline[linecolor=black, linewidth=0.04](0.88,-0.86634344)(0.22,-2.9863434)
\psline[linecolor=black, linewidth=0.04](1.28,-0.86634344)(2.74,-1.7463434)(2.74,0.95365655)
\psline[linecolor=black, linewidth=0.04](2.72,-1.7463434)(3.34,-2.9863434)
\pscircle[linecolor=black, linewidth=0.04, dimen=outer](8.08,1.1936566){0.32}
\pscircle[linecolor=black, linewidth=0.04, dimen=outer](6.56,-0.6463434){0.32}
\rput[bl](8.0,1.0736566){$\rho$}
\psline[linecolor=black, linewidth=0.04](6.1,2.9936566)(6.38,-0.4263434)
\psline[linecolor=black, linewidth=0.04](9.26,2.9736567)(8.28,1.3736566)
\psline[linecolor=black, linewidth=0.04](7.9,0.9736566)(6.78,-0.44634342)
\psline[linecolor=black, linewidth=0.04](6.36,-0.86634344)(5.7,-2.9863434)
\psline[linecolor=black, linewidth=0.04](6.76,-0.86634344)(8.22,-1.7463434)(8.22,0.95365655)
\psline[linecolor=black, linewidth=0.04](8.2,-1.7463434)(8.82,-2.9863434)
\rput[bl](4.6,0.11365658){$=$}
\rput[bl](6.42,-0.80634344){$\psi$}
\rput[bl](0.58,2.4536567){$T$}
\rput[bl](3.66,2.4336567){$F$}
\rput[bl](0.0,-2.6263435){$G$}
\rput[bl](8.74,-2.6863434){$S$}
\rput[bl](9.16,2.4136565){$F$}
\rput[bl](1.46,0.27365658){$F$}
\rput[bl](5.74,2.4536567){$T$}
\rput[bl](3.24,-2.6663435){$S$}
\rput[bl](5.46,-2.6663435){$G$}
\rput[bl](6.96,0.23365659){$G$}
\end{pspicture}
}
\end{center}     
 
Also in \cite{71}, wreaths were introduced and defined concisely as
monads in the 2-category $\mathrm{EM}(\mathfrak{K})$. 
The wreath product is the monad obtained as the Eilenberg-Moore construction 
in $\mathrm{EM}(\mathfrak{K})$ on the wreath.
Indeed, as always with a completion under limits, 
$\mathrm{EM}$ is a (dual) Kock-Z\"oberlein monad \cite{Kock1995, FiB} on the 2-category
of 2-categories while the wreath product 
\begin{eqnarray*}
\mathrm{wr}= \mathrm{wr}_{\mathfrak{K}}\dd \mathrm{EM}(\mathrm{EM}(\mathfrak{K}))\lra \mathrm{EM}(\mathfrak{K})
\end{eqnarray*}
gives the multiplication for that monad; the unit
\begin{eqnarray*}
\mathrm{id}=\mathrm{id}_{\mathfrak{K}}\dd \mathfrak{K}\lra \mathrm{EM}(\mathfrak{K})
\end{eqnarray*}
simply takes each object $\CA$ to the identity monad on $\CA$.
When it exists, the Eilenberg-Moore construction for a 2-category $\mathfrak{K}$
is a right adjoint
\begin{eqnarray*}
\mathrm{em}=\mathrm{em}_{\mathfrak{K}}\dd \mathrm{EM}(\mathfrak{K}) \lra \mathfrak{K}
\end{eqnarray*}
to $\mathrm{id}$; we put $\CA^{T}= \mathrm{em}_{\CA}T= \mathrm{em}(\CA,T)$.

We will now describe wreaths explicitly using string diagrams.

Let $T=(T,\mu, \eta)$ be a monad on an object $\CA$ of $\CK$. 

A {\em wreath around} $T$ consists of an endomorphism $S$
on $\CA$, and 2-cells $\nu\dd SS\Lra ST$,
$\sigma\dd 1_{\CA}\Lra ST$ and $\lambda\dd TS\Lra ST$
satisfying seven axioms.   

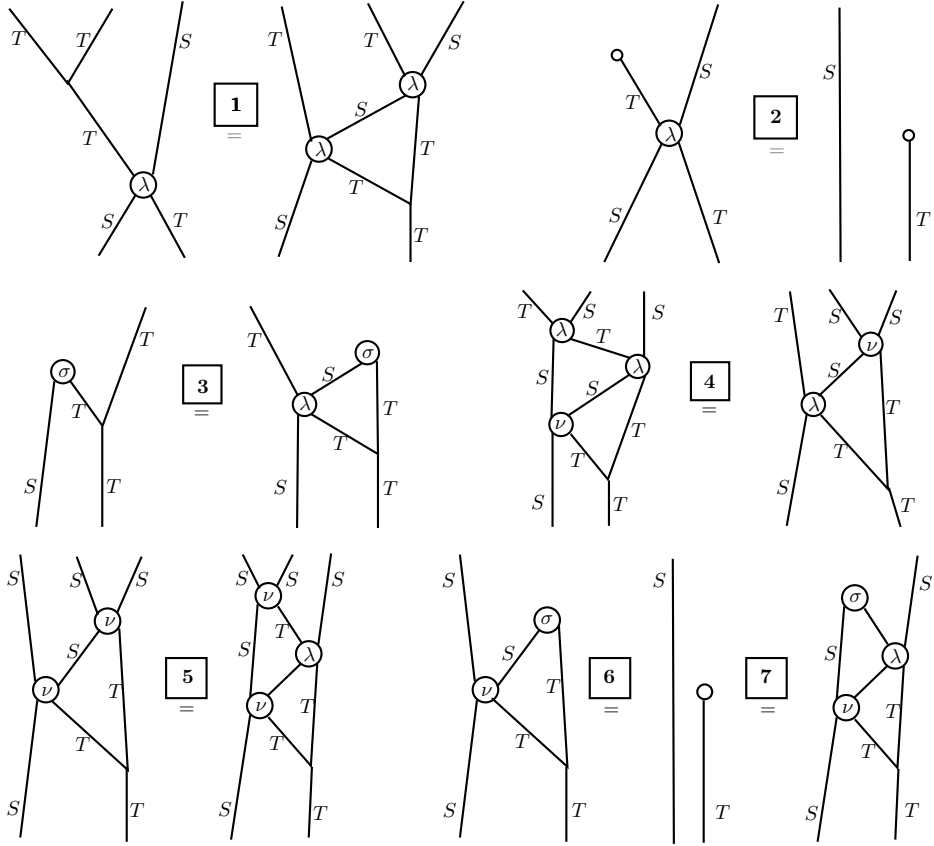
\begin{center}
\psscalebox{0.7 0.7} 
{
\begin{pspicture}(0,-2.4881215)(17.315805,2.4881215)
\pscircle[linecolor=black, linewidth=0.04, dimen=outer](2.5358045,-1.0017898){0.26}
\pscircle[linecolor=black, linewidth=0.04, dimen=outer](5.8358045,-0.32178977){0.26}
\pscircle[linecolor=black, linewidth=0.04, dimen=outer](7.5958047,0.8982102){0.26}
\rput[bl](4.0958047,-0.16178976){=}
\psline[linecolor=black, linewidth=0.04](1.0958046,0.9782102)(2.3558047,-0.82178974)
\psline[linecolor=black, linewidth=0.04](2.3958046,-1.1817898)(1.6958046,-2.3417897)
\psline[linecolor=black, linewidth=0.04](2.6758046,-1.2017897)(3.3158045,-2.3817897)
\psline[linecolor=black, linewidth=0.04](2.7158046,-0.80178976)(3.2758045,2.4782102)
\psline[linecolor=black, linewidth=0.04](7.455805,1.0582103)(6.7558045,2.4382102)
\psline[linecolor=black, linewidth=0.04](7.7558045,1.0782102)(8.555804,2.4782102)
\psline[linecolor=black, linewidth=0.04](7.7158046,0.67821026)(7.5558047,-1.3817898)
\psline[linecolor=black, linewidth=0.04](5.1358047,2.3782103)(5.6958046,-0.18178976)
\psline[linecolor=black, linewidth=0.04](5.995805,-0.12178976)(7.475805,0.69821024)
\psline[linecolor=black, linewidth=0.04](6.015805,-0.50178975)(7.5558047,-1.3617897)
\psline[linecolor=black, linewidth=0.04](5.6958046,-0.54178977)(5.0758047,-2.3617897)
\psline[linecolor=black, linewidth=0.04](0.015804596,2.3382103)(1.1158046,0.91821027)(1.9558046,2.3382103)(1.9558046,2.3382103)
\psline[linecolor=black, linewidth=0.04](7.5558047,-1.3617897)(7.5558047,-2.4617898)
\rput[bl](7.495805,0.79821026){$\lambda$}
\rput[bl](5.7558045,-0.46178976){$\lambda$}
\rput[bl](2.4358046,-1.1217898){$\lambda$}
\rput[bl](1.3958046,-0.24178976){$T$}
\rput[bl](1.7358046,-1.8217897){$S$}
\rput[bl](3.0958047,-1.8017898){$T$}
\rput[bl](4.8758044,1.6382103){$T$}
\rput[bl](6.6758046,1.6182102){$T$}
\rput[bl](7.6158047,-2.0817897){$T$}
\rput[bl](1.2758046,1.6182102){$T$}
\rput[bl](0.055804595,1.6582103){$T$}
\rput[bl](6.3958044,-1.1817898){$T$}
\rput[bl](7.7358046,-0.44178975){$T$}
\rput[bl](4.995805,-1.8417897){$S$}
\rput[bl](6.495805,0.29821023){$S$}
\rput[bl](8.255805,1.5782102){$S$}
\rput[bl](3.1958046,1.5982102){$S$}
\rput[bl](4.1758046,0.39821023){1}
\pscircle[linecolor=black, linewidth=0.04, dimen=outer](12.415805,-0.021789761){0.26}
\rput[bl](12.335805,-0.16178976){$\lambda$}
\psline[linecolor=black, linewidth=0.04](12.255805,0.11821024)(11.495805,1.3982103)
\pscircle[linecolor=black, linewidth=0.04, dimen=outer](11.435804,1.4582102){0.12}
\psline[linecolor=black, linewidth=0.04](12.615805,0.13821024)(13.335805,2.4182103)(13.335805,2.4182103)
\psline[linecolor=black, linewidth=0.04](12.2758045,-0.22178976)(11.195805,-2.4617898)
\psline[linecolor=black, linewidth=0.04](12.595804,-0.20178977)(13.355804,-2.4817898)
\rput[bl](14.295805,-0.46178976){=}
\psline[linecolor=black, linewidth=0.04](15.615805,2.3582103)(15.635804,-2.4617898)
\pscircle[linecolor=black, linewidth=0.04, dimen=outer](16.915804,-0.061789762){0.12}
\psline[linecolor=black, linewidth=0.04](16.935804,-0.16178976)(16.935804,-2.4417899)
\rput[bl](15.295805,0.99821025){$S$}
\rput[bl](12.975804,1.0182103){$S$}
\rput[bl](11.2758045,-1.7017897){$S$}
\rput[bl](11.595804,0.43821025){$T$}
\rput[bl](13.235805,-1.7417898){$T$}
\rput[bl](17.055805,-1.7817898){$T$}
\psframe[linecolor=black, linewidth=0.04, dimen=outer](14.835805,0.69821024)(13.995805,-0.14178976)
\psframe[linecolor=black, linewidth=0.04, dimen=outer](4.6958046,0.93821025)(3.8558047,0.09821024)
\rput[bl](4.1758046,0.39821023){$\textbf{1}$}
\rput[bl](4.1758046,0.39821023){$\textbf{1}$}
\rput[bl](14.3158045,0.17821024){$\textbf{2}$}
\end{pspicture}
}
\end{center}

\begin{center}
\psscalebox{0.7 0.7} 
{
\begin{pspicture}(0,-2.265635)(16.74,2.265635)
\pscircle[linecolor=black, linewidth=0.04, dimen=outer](0.74,0.6945491){0.24}
\pscircle[linecolor=black, linewidth=0.04, dimen=outer](10.12,1.474549){0.24}
\pscircle[linecolor=black, linewidth=0.04, dimen=outer](5.28,0.07454907){0.24}
\pscircle[linecolor=black, linewidth=0.04, dimen=outer](6.46,1.0545491){0.24}
\psline[linecolor=black, linewidth=0.04](0.88,0.5145491)(1.48,-0.32545093)
\psline[linecolor=black, linewidth=0.04](0.58,0.5145491)(0.24,-2.245451)
\psline[linecolor=black, linewidth=0.04](2.3,1.9145491)(1.48,-0.34545094)(1.48,-2.245451)
\psline[linecolor=black, linewidth=0.04](4.26,1.9345491)(5.16,0.23454908)
\psline[linecolor=black, linewidth=0.04](6.64,0.8945491)(6.66,-0.8654509)(5.4,-0.10545093)
\psline[linecolor=black, linewidth=0.04](5.4,0.25454906)(6.36,0.85454905)
\psline[linecolor=black, linewidth=0.04](6.66,-0.88545096)(6.66,-2.265451)
\rput[bl](3.2,-0.18545093){$=$}
\psframe[linecolor=black, linewidth=0.04, dimen=outer](3.74,0.83454907)(2.98,0.07454907)
\rput[bl](3.26,0.29454908){$\textbf{3}$}
\rput[bl](6.36,0.9545491){$\sigma$}
\rput[bl](5.18,-0.02545093){$\lambda$}
\rput[bl](0.66,0.59454906){$\sigma$}
\rput[bl](10.0,1.374549){$\lambda$}
\rput[bl](11.44,0.67454904){$\lambda$}
\rput[bl](14.74,-0.04545093){$\lambda$}
\psline[linecolor=black, linewidth=0.04](5.16,-0.12545092)(5.14,-2.265451)
\rput[bl](1.56,-1.605451){$T$}
\rput[bl](0.0,-1.605451){$S$}
\rput[bl](2.18,1.1745491){$T$}
\rput[bl](4.2,1.1545491){$T$}
\rput[bl](0.9,-0.18545093){$T$}
\rput[bl](5.6,0.5145491){$S$}
\rput[bl](4.7,-1.625451){$S$}
\rput[bl](6.76,-1.645451){$T$}
\rput[bl](5.82,-0.78545094){$T$}
\rput[bl](6.76,-0.12545092){$T$}
\pscircle[linecolor=black, linewidth=0.04, dimen=outer](11.54,0.79454905){0.24}
\pscircle[linecolor=black, linewidth=0.04, dimen=outer](10.1,-0.30545092){0.24}
\pscircle[linecolor=black, linewidth=0.04, dimen=outer](15.92,1.2145491){0.24}
\pscircle[linecolor=black, linewidth=0.04, dimen=outer](14.84,0.07454907){0.24}
\psline[linecolor=black, linewidth=0.04](10.28,-0.48545092)(10.98,-1.3454509)(11.68,0.61454904)
\psline[linecolor=black, linewidth=0.04](11.0,-1.365451)(11.0,-2.225451)
\psline[linecolor=black, linewidth=0.04](9.96,1.3145491)(9.92,-0.18545093)
\psline[linecolor=black, linewidth=0.04](10.28,1.3145491)(11.4,0.9545491)
\psline[linecolor=black, linewidth=0.04](9.38,2.234549)(9.94,1.614549)
\psline[linecolor=black, linewidth=0.04](10.28,1.634549)(10.66,2.214549)
\psline[linecolor=black, linewidth=0.04](11.66,0.9545491)(11.66,2.214549)
\psline[linecolor=black, linewidth=0.04](11.38,0.6345491)(10.22,-0.14545093)
\psline[linecolor=black, linewidth=0.04](9.94,-0.46545094)(9.94,-2.245451)
\rput[bl](12.74,-0.12545092){$=$}
\psframe[linecolor=black, linewidth=0.04, dimen=outer](13.28,0.8945491)(12.52,0.13454907)
\rput[bl](12.78,0.39454907){$\textbf{4}$}
\psline[linecolor=black, linewidth=0.04](14.98,-0.12545092)(16.24,-1.525451)(16.1,1.0945491)
\psline[linecolor=black, linewidth=0.04](15.74,1.354549)(15.14,2.254549)
\psline[linecolor=black, linewidth=0.04](16.06,1.374549)(16.4,2.234549)
\psline[linecolor=black, linewidth=0.04](15.78,1.0345491)(14.94,0.23454908)
\psline[linecolor=black, linewidth=0.04](14.68,0.21454906)(14.4,2.214549)
\psline[linecolor=black, linewidth=0.04](14.72,-0.08545093)(14.34,-2.225451)
\psline[linecolor=black, linewidth=0.04](16.26,-1.505451)(16.48,-2.245451)
\rput[bl](15.84,1.0945491){$\nu$}
\rput[bl](9.98,-0.38545093){$\nu$}
\rput[bl](9.28,1.7145491){$T$}
\rput[bl](10.5,1.7145491){$S$}
\rput[bl](10.28,-1.105451){$T$}
\rput[bl](10.76,1.234549){$T$}
\rput[bl](11.42,-0.5454509){$T$}
\rput[bl](11.08,-1.9454509){$T$}
\rput[bl](16.48,-1.9654509){$T$}
\rput[bl](9.62,0.45454907){$S$}
\rput[bl](11.8,1.6945491){$S$}
\rput[bl](10.56,0.25454906){$S$}
\rput[bl](9.54,-1.9054509){$S$}
\rput[bl](14.14,-1.625451){$S$}
\rput[bl](15.1,0.59454906){$S$}
\rput[bl](16.26,1.5945491){$S$}
\rput[bl](15.12,1.614549){$S$}
\rput[bl](15.22,-0.9454509){$T$}
\rput[bl](16.2,-0.10545093){$T$}
\rput[bl](14.1,1.634549){$T$}
\end{pspicture}
}
\end{center}

\begin{center}
\psscalebox{0.7 0.7} 
{
\begin{pspicture}(0,-2.7514224)(17.38,2.7514224)
\pscircle[linecolor=black, linewidth=0.04, dimen=outer](0.74,0.16866995){0.26}
\pscircle[linecolor=black, linewidth=0.04, dimen=outer](1.9,1.4686699){0.26}
\pscircle[linecolor=black, linewidth=0.04, dimen=outer](4.92,1.9486699){0.26}
\pscircle[linecolor=black, linewidth=0.04, dimen=outer](4.76,-0.13133006){0.26}
\pscircle[linecolor=black, linewidth=0.04, dimen=outer](10.16,1.48867){0.26}
\pscircle[linecolor=black, linewidth=0.04, dimen=outer](5.68,0.84866995){0.26}
\psline[linecolor=black, linewidth=0.04](2.26,-1.2913301)(2.26,-2.71133)
\psline[linecolor=black, linewidth=0.04](1.72,1.5886699)(1.32,2.68867)
\psline[linecolor=black, linewidth=0.04](2.08,1.62867)(2.54,2.68867)
\psline[linecolor=black, linewidth=0.04](0.54,0.32866994)(0.26,2.72867)
\psline[linecolor=black, linewidth=0.04](0.96,0.22866994)(1.76,1.3086699)
\psline[linecolor=black, linewidth=0.04](0.58,-0.031330053)(0.26,-2.63133)
\psline[linecolor=black, linewidth=0.04](4.76,2.08867)(4.44,2.72867)
\psline[linecolor=black, linewidth=0.04](5.08,2.12867)(5.38,2.72867)
\psline[linecolor=black, linewidth=0.04](5.54,1.0686699)(5.1,1.74867)
\psline[linecolor=black, linewidth=0.04](5.86,1.02867)(6.1,2.7486699)
\psline[linecolor=black, linewidth=0.04](5.54,0.66866994)(4.9,0.04866995)
\psline[linecolor=black, linewidth=0.04](4.72,1.78867)(4.58,0.028669948)
\psline[linecolor=black, linewidth=0.04](4.58,-0.31133005)(4.22,-2.67133)
\psline[linecolor=black, linewidth=0.04](4.92,-0.35133004)(5.72,-1.27133)(5.84,0.66866994)
\psline[linecolor=black, linewidth=0.04](5.74,-1.2913301)(5.68,-2.63133)
\rput[bl](3.24,-0.29133004){$=$}
\rput[bl](15.84,1.8286699){$\sigma$}
\rput[bl](10.04,1.40867){$\sigma$}
\rput[bl](1.84,1.36867){$\nu$}
\rput[bl](0.66,0.04866995){$\nu$}
\rput[bl](4.82,1.84867){$\nu$}
\rput[bl](4.68,-0.23133005){$\nu$}
\rput[bl](5.58,0.70866996){$\lambda$}
\psframe[linecolor=black, linewidth=0.04, dimen=outer](3.78,0.78866994)(2.98,-0.011330051)
\rput[bl](3.3,0.28866994){$\textbf{5}$}
\pscircle[linecolor=black, linewidth=0.04, dimen=outer](15.94,1.90867){0.26}
\pscircle[linecolor=black, linewidth=0.04, dimen=outer](15.78,-0.17133005){0.26}
\pscircle[linecolor=black, linewidth=0.04, dimen=outer](16.7,0.8086699){0.26}
\psline[linecolor=black, linewidth=0.04](16.56,1.02867)(16.12,1.7086699)
\psline[linecolor=black, linewidth=0.04](16.88,0.98866993)(17.12,2.70867)
\psline[linecolor=black, linewidth=0.04](16.56,0.62867)(15.92,0.008669948)
\psline[linecolor=black, linewidth=0.04](15.74,1.74867)(15.6,-0.011330051)
\psline[linecolor=black, linewidth=0.04](15.6,-0.35133004)(15.24,-2.71133)
\psline[linecolor=black, linewidth=0.04](15.94,-0.39133006)(16.74,-1.3113301)(16.86,0.62867)
\psline[linecolor=black, linewidth=0.04](16.76,-1.3313301)(16.7,-2.67133)
\rput[bl](15.7,-0.27133006){$\nu$}
\rput[bl](16.6,0.66866994){$\lambda$}
\rput[bl](11.22,-0.31133005){$=$}
\psframe[linecolor=black, linewidth=0.04, dimen=outer](11.74,0.78866994)(10.94,-0.011330051)
\psline[linecolor=black, linewidth=0.04](0.86,-0.051330052)(2.28,-1.35133)(2.12,1.3286699)
\pscircle[linecolor=black, linewidth=0.04, dimen=outer](9.0,0.16866995){0.26}
\psline[linecolor=black, linewidth=0.04](10.52,-1.2913301)(10.52,-2.71133)
\psline[linecolor=black, linewidth=0.04](8.8,0.32866994)(8.52,2.72867)
\psline[linecolor=black, linewidth=0.04](9.22,0.22866994)(10.02,1.3086699)
\psline[linecolor=black, linewidth=0.04](8.84,-0.031330053)(8.52,-2.63133)
\rput[bl](8.92,0.04866995){$\nu$}
\psline[linecolor=black, linewidth=0.04](9.12,0.008669948)(10.54,-1.2913301)(10.38,1.38867)
\rput[bl](14.16,-0.31133005){$=$}
\psframe[linecolor=black, linewidth=0.04, dimen=outer](14.68,0.78866994)(13.88,-0.011330051)
\rput[bl](11.22,0.28866994){$\textbf{6}$}
\rput[bl](14.16,0.28866994){$\textbf{7}$}
\psline[linecolor=black, linewidth=0.04](12.52,2.64867)(12.54,-2.7513301)
\pscircle[linecolor=black, linewidth=0.04, dimen=outer](13.12,0.12866995){0.16}
\psline[linecolor=black, linewidth=0.04](13.1,-0.031330053)(13.1,-2.7313302)
\rput[bl](0.0,2.14867){$S$}
\rput[bl](1.3,-0.99133){$T$}
\rput[bl](1.12,0.76866996){$S$}
\rput[bl](0.0,-2.21133){$S$}
\rput[bl](1.18,2.10867){$S$}
\rput[bl](1.92,0.008669948){$T$}
\rput[bl](2.32,-2.2513301){$T$}
\rput[bl](4.94,-0.97133005){$T$}
\rput[bl](5.52,-0.27133006){$T$}
\rput[bl](5.76,-2.29133){$T$}
\rput[bl](5.04,1.14867){$T$}
\rput[bl](10.62,-2.2513301){$T$}
\rput[bl](2.42,2.12867){$S$}
\rput[bl](4.3,2.14867){$S$}
\rput[bl](5.24,2.12867){$S$}
\rput[bl](6.1,2.12867){$S$}
\rput[bl](9.3,0.72866994){$S$}
\rput[bl](8.18,2.10867){$S$}
\rput[bl](3.98,-2.2313302){$S$}
\rput[bl](4.34,0.8086699){$S$}
\rput[bl](12.14,2.10867){$S$}
\rput[bl](15.36,0.72866994){$S$}
\rput[bl](17.14,2.04867){$S$}
\rput[bl](8.2,-2.2513301){$S$}
\rput[bl](9.56,-0.93133){$T$}
\rput[bl](10.14,0.10866995){$T$}
\rput[bl](16.52,-0.37133005){$T$}
\rput[bl](16.06,-1.13133){$T$}
\rput[bl](13.32,-2.27133){$T$}
\rput[bl](16.88,-2.2513301){$T$}
\rput[bl](15.0,-2.29133){$S$}
\end{pspicture}
}
\end{center}

The {\em product} of the wreath $S$ around $T$, or the {\em wreath product}, is the monad consisting of the endomorphism $ST$ on $\CA$ with the multiplication and unit as displayed in the diagram:
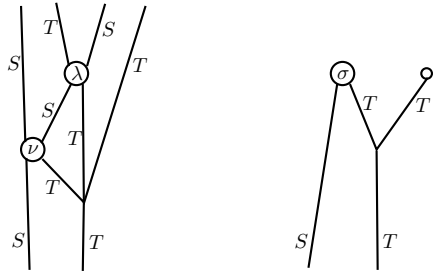
\begin{center}
\psscalebox{0.7 0.7} 
{
\begin{pspicture}(0,-2.5421195)(8.0,2.5421195)
\pscircle[linecolor=black, linewidth=0.04, dimen=outer](1.3,1.1981876){0.24}
\pscircle[linecolor=black, linewidth=0.04, dimen=outer](0.48,-0.24181244){0.24}
\pscircle[linecolor=black, linewidth=0.04, dimen=outer](6.3,1.1981876){0.24}
\pscircle[linecolor=black, linewidth=0.04, dimen=outer](7.88,1.1981876){0.12}
\psline[linecolor=black, linewidth=0.04](6.44,0.99818754)(6.94,-0.24181244)(7.9,1.1181875)
\psline[linecolor=black, linewidth=0.04](6.2,0.99818754)(5.66,-2.4818125)
\psline[linecolor=black, linewidth=0.04](6.94,-0.26181245)(6.96,-2.5218124)
\psline[linecolor=black, linewidth=0.04](1.16,1.3381876)(0.92,2.5381875)
\psline[linecolor=black, linewidth=0.04](1.5,1.3381876)(1.84,2.5181875)
\psline[linecolor=black, linewidth=0.04](1.2,0.99818754)(0.64,-0.12181244)
\psline[linecolor=black, linewidth=0.04](0.34,-0.101812445)(0.26,2.4781876)
\psline[linecolor=black, linewidth=0.04](0.36,-0.40181243)(0.42,-2.5218124)
\psline[linecolor=black, linewidth=0.04](0.66,-0.42181244)(1.44,-1.2418125)(1.42,0.99818754)
\psline[linecolor=black, linewidth=0.04](1.44,-1.2218125)(1.42,-2.5418124)
\psline[linecolor=black, linewidth=0.04](2.6,2.4781876)(1.44,-1.2618124)
\rput[bl](0.72,-1.0818125){$T$}
\rput[bl](0.0,1.2581875){$S$}
\rput[bl](1.14,-0.16181244){$T$}
\rput[bl](0.68,1.9381876){$T$}
\rput[bl](2.36,1.2181876){$T$}
\rput[bl](1.54,-2.1218123){$T$}
\rput[bl](0.64,0.37818757){$S$}
\rput[bl](0.08,-2.0818124){$S$}
\rput[bl](1.78,1.9181875){$S$}
\rput[bl](5.4,-2.1018124){$S$}
\rput[bl](6.68,0.49818754){$T$}
\rput[bl](7.68,0.47818756){$T$}
\rput[bl](7.06,-2.1018124){$T$}
\rput[bl](1.2,1.0781876){$\lambda$}
\rput[bl](0.38,-0.32181245){$\nu$}
\rput[bl](6.2,1.0981876){$\sigma$}
\end{pspicture}
}
\end{center}
where the unlabelled nodes are the ternary and binary multiplications of $T$.

A {\em distributive law} \cite{Beck1969} of a monad $T$ over a 
monad $S$ is a special case of a wreath around $T$ consisting of the 
endofunctor $S$ while the natural transformations 
$\nu$ and $\sigma$ of the special form
\begin{eqnarray}\label{distlawnusig}
\nu = \left(SS\stackrel{\mu}\lra S\stackrel{S\eta}\lra ST\right) \ \text{ and } \ \sigma = \left(1_{\CA}\stackrel{\eta \eta}\lra ST\right) \ ,
\end{eqnarray}
and $\lambda$ remains arbitrary.

 \begin{example}\label{exwreathmonoid} 
 {\em We now generalise Example 3.2 of \cite{71} from groups to monoids. 
 We call a monoid morphism $p\dd E\to M$ (in the category $\mathrm{Set}$ of sets)
 a {\em normal cloven lax fibration} when it is equipped with a function $j\dd M\to E$ such that 
 $p\circ j = 1_M$, $j(1)=1$ and, for
 \begin{eqnarray*}
A=\{ a\in E \dd p(a)=1\} = p^{-1}(1) \ ,
\end{eqnarray*}
the function $h\dd M\times A\to E$, defined by $h(x,a)= j(x)a$, is invertible.
This gives, for each $x\in M$, a pullback square.
 \begin{eqnarray}\label{allfibres}
  \begin{aligned}
\xymatrix{
A \ar[rr]^-{h(x,-)} \ar[d]_-{!} && E \ar[d]^-{p} \\
1 \ar[rr]_-{x} && M}
\end{aligned}
\end{eqnarray}  
Generally, the kernel of a monoid morphism is a rather strange thing to consider,
yet, because we have a fibration, all the fibres $p^{-1}(x)$ of $p$ are isomorphic
as sets.
Unlike arbitrary fibres, the kernel has the advantage of being a submonoid of $E$.

We use the pullback \eqref{allfibres} to obtain a function $\alpha \dd A\times M\to A$ of
$M$ on $A$; indeed, $\alpha(a,x)=a\cdot x\in A$ is characterized by the
property
\begin{eqnarray}\label{dot}
j(x)(a\cdot x)=aj(x) \ .
\end{eqnarray}
In other words, this $\alpha$ measures the failure of the kernel to commute with the image of $j$.
Using the pullback uniqueness clause, we see that each $-\cdot x\dd A\to A$ is a monoid
morphism.

We also use the pullback \eqref{allfibres} with $x$ replaced by $xy$ to obtain a function
$\rho\dd M\times M\to A$ characterized by the property
\begin{eqnarray}\label{rho}
j(xy)\rho(x,y)=j(x)j(y) \ .
\end{eqnarray}
In other words, $\rho$ measures the failure of $j$ to be a monoid morphism.
Indeed, for each $x,y\in M$, we have a 2-cell
\begin{equation}\label{rhodiag}
\begin{aligned}
\xymatrix{
A \ar[rd]_{-\cdot (xy)}^(0.5){\phantom{a}}="1" \ar[rr]^{-\cdot x}  && A \ar[ld]^{-\cdot y}_(0.5){\phantom{a}}="2" \ar@{<=}"1";"2"^-{\rho(x,y)}
\\
& A 
}
\end{aligned}
\end{equation}
in the 2-category $\mathrm{Mon}=\mathrm{MonSet}$ of monoids; the ``naturality'' amounts to
the equation
\begin{eqnarray}\label{rho-action}
(a\cdot (xy))\rho(x,y)=\rho(x,y)((a\cdot x)\cdot y)
\end{eqnarray}
which shows that $\rho$ also measures the failure of $\alpha$ to be an action of the monoid $M$ on $A$.
To prove \eqref{rho-action}, it suffices to prove we have equality after applying $h(xy,-)$, which we do thus:
\begin{eqnarray*}
j(xy)(a\cdot (xy))\rho(x,y) & = & aj(xy)\rho(x,y) \\
& = & a j(x)j(y) \\
&=& j(x)(a\cdot x)j(y) \\
&=&j(x)j(y)((a\cdot x)\cdot y) \\
&=& j(xy)\rho(x,y)((a\cdot x)\cdot y) \ .
\end{eqnarray*}

Let $\Sigma M$ denote the category with one object $0$ and hom
$\Sigma M(0,0)=M$; composition is multiplication in $M$.
What we are producing is a normal lax functor
\begin{eqnarray}
P\dd \Sigma M^{\mathrm{op}} \lra \mathrm{Mon}
\end{eqnarray}
with $P0=A$ and $Px=-\cdot x\dd A\to A$.
The composition constraints are given by \eqref{rhodiag}.
Clearly $\rho(1,x)=1=\rho(x,1)$ so all that remains to prove is the
coherence condition \eqref{rhocoher}.

\begin{equation}\label{rhocoher}
\begin{aligned}
  \cd{
    {A}  \ar[rr]^-{-\cdot y}  \dtwocell{dr}{\rho(x,y)} & & {A} \ar[dd]^{-\cdot z} \\
    && \\
    {A} \ar[rr]_-{-\cdot (xyz)} \ar[uurr]^-{} \ar[uu]^{-\cdot x} & & {A} \dtwocell{ull}{\rho(xy,z)} }
    \qquad
        \cd{ = }
    \qquad 
     \cd{
    {A} \ar[ddrr]^-{} \ar[rr]^-{-\cdot y}   & & {A} \ar[dd]^{-\cdot z} \dtwocell{dll}{\rho(y,z)}\\
    && \\
    {A} \dtwocell{ur}{\rho(x,yz)} \ar[rr]_-{-\cdot (xyz)}  \ar[uu]^{-\cdot x} & & {A} 
  }
  \end{aligned}
\end{equation}
This amounts to the Schreier factor set or 2-cocycle condition \eqref{factorset}.
\begin{eqnarray}\label{factorset}
\rho(xy,z)(\rho(x,y)\cdot z)=\rho(x,yz)\rho(y,z)
\end{eqnarray}
To prove this, it suffices to check after left multiplication by $j(xyz)$, which
we do thus:
\begin{eqnarray*}
j(xyz)\rho(xy,z)(\rho(x,y)\cdot z) & = & j(xy)j(z)(\rho(x,y)\cdot z) \\
& = & j(xy)\rho(x,y)j(z) \\
&=& j(x)j(y)j(z) \\
&=&j(x)j(yz)\rho(y,z) \\
&=& j(xyz)\rho(x,yz)\rho(y,z) \ .
\end{eqnarray*}

We are now in a position to transport the multiplication of $E$ to $M\times A$
across the isomorphism $h$.
\begin{equation*}
\xymatrix{
M\times A \ar[rd]_{\mathrm{pr}_1}\ar[rr]^{h}   && E \ar[ld]^{p} \\
& M  &
}
\end{equation*}
The resultant multiplication on $M\times A$ is
\begin{eqnarray}\label{rhoproduct}
(x,a)(y,b)=(xy,\rho(x,y)(a\cdot y)b) \ .
\end{eqnarray}

This gives an equivalence of categories between normal cloven lax fibrations
over any monoid $M$ and normal lax functors 
$P\dd \Sigma M^{\mathrm{op}} \lra \mathrm{Mon}$. 
This is essentially classical and is an interpretation theorem
for the second cohomology of the monoid $M$: 2-cocycles equate
to certain extensions $E\to M$.

Now we give the wreath. The category is $\mathrm{Set}$.
The remaining data all arise from data in $\mathrm{Set}$
by applying the strong monoidal functor 
$\mathrm{Set}\to [\mathrm{Set},\mathrm{Set}]$
which takes $K$ to $K\times -$. 
The monad $T$ arises from the monoid $A$.
The endofunctor $S$ arises from the set $M$.
The natural transformation $\nu$ arises from the function 
$M\times M\to M\times A, (x,y)\mapsto (xy,\rho(x,y))$.
The natural transformation $\lambda$ arises from the function 
$A\times M\to M\times A, (a,x)\mapsto (x,\alpha(a,x))$.
The natural transformation $\sigma$ arises from the function 
$\mathbf{1}\to M\times A$ which picks out $(1,1)$.

The wreath product of course arises from the monoid $M\times A$ 
with product \eqref{rhoproduct} and so recaptures $E$ up to isomorphism.}
 \end{example} 

\begin{remark}
{\em 
\begin{enumerate}
\item Here is the string diagram for \eqref{rho-action}.
\begin{center}
\begin{equation}\label{rhonaturality}
\begin{aligned}
\psscalebox{0.55 0.55} 
{
\begin{pspicture}(0,-3.6962144)(19.88,3.6962144)
\rput[bl](3.44,-0.395969){$\alpha$}
\pscircle[linecolor=black, linewidth=0.04, dimen=outer](3.52,-0.215969){0.54}
\psline[linecolor=black, linewidth=0.04](0.4,3.684031)(3.12,0.144031)
\psline[linecolor=black, linewidth=0.04](3.6,3.644031)(3.56,0.26403102)
\psline[linecolor=black, linewidth=0.04](5.74,3.624031)(3.58,0.924031)
\psline[linecolor=black, linewidth=0.04](3.6,2.684031)(4.42,2.164031)(4.42,2.164031)
\pscircle[linecolor=black, linewidth=0.04, dimen=outer](6.14,1.204031){0.5}
\rput[bl](6.04,1.064031){$\rho$}
\psline[linecolor=black, linewidth=0.04](4.68,2.044031)(5.84,1.564031)
\psline[linecolor=black, linewidth=0.04](5.14,2.884031)(6.34,1.644031)
\psline[linecolor=black, linewidth=0.04](3.54,-0.755969)(4.54,-1.935969)
\psline[linecolor=black, linewidth=0.04](6.14,0.72403103)(4.56,-1.935969)(4.58,-3.695969)
\rput[bl](2.9,3.144031){$M$}
\rput[bl](0.0,3.164031){$A$}
\rput[bl](5.64,3.124031){$M$}
\rput[bl](8.28,0.664031){$\Huge{=}$}
\pscircle[linecolor=black, linewidth=0.04, dimen=outer](12.72,1.144031){0.5}
\pscircle[linecolor=black, linewidth=0.04, dimen=outer](16.66,1.1240311){0.5}
\pscircle[linecolor=black, linewidth=0.04, dimen=outer](18.38,-0.23596899){0.5}
\psline[linecolor=black, linewidth=0.04](16.7,0.664031)(18.04,0.10403101)
\rput[bl](16.62,0.984031){$\alpha$}
\rput[bl](18.32,-0.375969){$\alpha$}
\rput[bl](12.68,1.0040311){$\rho$}
\psline[linecolor=black, linewidth=0.04](18.98,3.604031)(18.7,0.10403101)
\psline[linecolor=black, linewidth=0.04](18.94,2.864031)(12.96,1.544031)(13.02,1.5040311)
\psline[linecolor=black, linewidth=0.04](15.22,3.604031)(15.22,3.004031)(12.5,1.584031)
\psline[linecolor=black, linewidth=0.04](15.22,2.984031)(16.22,2.464031)(16.24,2.464031)
\psline[linecolor=black, linewidth=0.04](16.54,2.224031)(16.98,1.484031)
\psline[linecolor=black, linewidth=0.04](12.14,3.624031)(14.06,2.584031)(14.06,2.604031)
\psline[linecolor=black, linewidth=0.04](14.34,2.444031)(14.9,2.144031)
\psline[linecolor=black, linewidth=0.04](15.32,1.904031)(16.36,1.524031)
\psline[linecolor=black, linewidth=0.04](12.72,0.684031)(15.46,-1.535969)
\psline[linecolor=black, linewidth=0.04](18.36,-0.735969)(15.48,-1.575969)(15.48,-3.675969)
\rput[bl](4.74,-3.395969){$A$}
\rput[bl](15.66,-3.515969){$A$}
\rput[bl](11.6,3.164031){$A$}
\rput[bl](15.32,3.204031){$M$}
\rput[bl](19.12,3.204031){$M$}
\end{pspicture}
}
\end{aligned}
\end{equation}
\end{center}
\item Here is the string diagram for \eqref{factorset}.
\begin{center}
\begin{equation}
\begin{aligned}
\psscalebox{0.55 0.55} 
{
\begin{pspicture}(0,-3.377788)(16.196299,3.377788)
\pscircle[linecolor=black, linewidth=0.04, dimen=outer](1.7762985,-0.037788164){0.42}
\pscircle[linecolor=black, linewidth=0.04, dimen=outer](4.3562984,1.1622119){0.42}
\pscircle[linecolor=black, linewidth=0.04, dimen=outer](5.2362986,-0.5177882){0.42}
\psline[linecolor=black, linewidth=0.04](5.7962985,3.362212)(1.9962986,0.30221185)
\psline[linecolor=black, linewidth=0.04](2.9362986,3.3222117)(1.5162985,0.24221183)
\psline[linecolor=black, linewidth=0.04](0.016298523,3.3022118)(1.7962985,0.8022118)
\psline[linecolor=black, linewidth=0.04](2.6562986,2.642212)(3.8762984,1.9822118)
\psline[linecolor=black, linewidth=0.04](4.1362987,1.8822118)(4.6562986,1.4422119)
\psline[linecolor=black, linewidth=0.04](0.63629854,2.422212)(2.1162984,2.0022118)
\psline[linecolor=black, linewidth=0.04](2.4362986,1.9222119)(3.4762986,1.6822119)
\psline[linecolor=black, linewidth=0.04](3.8162985,1.6022118)(4.1362987,1.5022118)
\psline[linecolor=black, linewidth=0.04](4.5562987,0.8022118)(5.1162987,-0.15778816)
\psline[linecolor=black, linewidth=0.04](5.2762985,2.922212)(5.4162984,-0.17778817)
\psline[linecolor=black, linewidth=0.04](1.8162985,-0.45778817)(3.0762985,-1.7377882)(3.0762985,-3.357788)
\psline[linecolor=black, linewidth=0.04](3.0962985,-1.7577882)(5.1962986,-0.89778817)
\rput[bl](14.176298,0.82221186){$\rho$}
\rput[bl](10.976298,-0.17778817){$\rho$}
\rput[bl](4.1962986,1.0622119){$\rho$}
\rput[bl](1.6362985,-0.15778816){$\rho$}
\rput[bl](5.1362987,-0.6977882){$\alpha$}
\rput[bl](7.6762986,0.16221184){$\Huge{=}$}
\pscircle[linecolor=black, linewidth=0.04, dimen=outer](11.116299,-0.017788162){0.4}
\pscircle[linecolor=black, linewidth=0.04, dimen=outer](14.376299,1.0022118){0.4}
\psline[linecolor=black, linewidth=0.04](9.896298,3.3222117)(10.916299,0.26221183)
\psline[linecolor=black, linewidth=0.04](15.436298,3.3422117)(14.616299,1.2622118)
\psline[linecolor=black, linewidth=0.04](15.216298,2.7022119)(11.296299,0.30221185)
\psline[linecolor=black, linewidth=0.04](12.416299,3.3422117)(12.336299,0.9222118)
\psline[linecolor=black, linewidth=0.04](12.436298,2.662212)(13.556298,1.8222119)(13.536299,1.8222119)
\psline[linecolor=black, linewidth=0.04](13.716298,1.6822119)(14.156299,1.3022119)
\psline[linecolor=black, linewidth=0.04](11.196299,-0.39778817)(13.056298,-1.7177882)
\psline[linecolor=black, linewidth=0.04](14.416299,0.60221183)(13.036299,-1.7177882)(13.036299,-3.377788)
\rput[bl](0.3162985,2.882212){$M$}
\rput[bl](3.2162986,-3.1777883){$A$}
\rput[bl](2.8962984,2.882212){$M$}
\rput[bl](5.6562986,2.862212){$M$}
\rput[bl](10.096298,2.9822118){$M$}
\rput[bl](15.436298,2.9822118){$M$}
\rput[bl](12.516298,3.0022118){$M$}
\rput[bl](13.156299,-3.2177882){$A$}
\end{pspicture}
}
\end{aligned}
\end{equation}
\end{center}
\item
The structure on $p\dd E\to M$ of {\em normal cloven (strict) fibration} consists of
 a function $j\dd M\to E$ such that $p\circ j = 1_M$, $j(1)=1$ and the square
 \begin{eqnarray*}
\xymatrix{
M\times E \ar[rr]^-{\mu\circ (j\times 1_E)} \ar[d]_-{1_M\times p} && E \ar[d]^-{p} \\
M\times M \ar[rr]_-{\mu} && M}
\end{eqnarray*}
is a pullback.
Notice that we have the condition $M\times A\cong E$ for a lax fibration since we can paste two pullback squares as follows:
\begin{eqnarray*}
\xymatrix{
M\times A \ar[rr]^-{1_M \times \mathrm{incl} } \ar[d]_-{\mathrm{pr}_1}&& M\times E \ar[rr]^-{\mu\circ (j\times 1_E)} \ar[d]_-{1_M\times p} && E \ar[d]^-{p} \\
M \ar[rr]_-{(1_M,1)} \ar@/_2pc/[rrrr]_{1_M} && M\times M \ar[rr]_-{\mu} && M \ .}
\end{eqnarray*}
\item Similarly to Example~\ref{exwreathmonoid}, Example 3.3 of \cite{71} can be generalised from Hopf algebras $H$ to bimonoids $M$ in a braided monoidal category.
Moreover, there is no need for the convolution invertibility of $\rho$: however, the
one axiom required for $A$ to be a twisted $M$-module, which is stated in \cite{71} in terms
of that inverse of $\rho$, should be replaced by the naturality condition \eqref{rhonaturality}.
We will discuss a dual of this in Section~\ref{sectiontwcoact}. 
\end{enumerate}}
\end{remark}

\section{Mixed wreaths}

There are several possibilities for mixed wreaths just as for mixed distributive laws; compare \cite{PoWa2002}.

We will use the notation 
\begin{eqnarray*}
\mathrm{EM}^{\mathrm{du}}(\mathfrak{K})=\mathrm{EM}(\mathfrak{K}^{\mathrm{du}})^{\mathrm{du}}
\end{eqnarray*}
for any of the dualities $\mathrm{du}\in \{\mathrm{op}, \mathrm{co}, \mathrm{coop}\}$ (in the notation of \cite{7}).
We also put 
\begin{eqnarray*}
\mathrm{KL}(\mathfrak{K}) = \mathrm{EM}^{\mathrm{op}}(\mathfrak{K})
\end{eqnarray*}
since it is the cocompletion of $\mathfrak{K}$ with respect to the Kleisli construction.
This then leads to 
\begin{eqnarray*}
\mathrm{KL}^{\mathrm{co}}(\mathfrak{K}) = \mathrm{EM}^{\mathrm{coop}}(\mathfrak{K}) \ .
\end{eqnarray*}

\begin{definition}
{\em Let $T$ be a monad on $\CA$ in the 2-category $\mathfrak{K}$.
A {\em mixed wreath around the monad} $T$ is a comonad on $(\CA,T)$ 
in the 2-category $\mathrm{EM}(\mathfrak{K})$.}  
\end{definition}

More explicitly, a mixed wreath structure around $T$ on an endomorphism $G$ of $\CA$ consists of
2-cells $\delta \dd G\Rightarrow GGT$, $\varepsilon \dd G\Rightarrow T$ and $\xi \dd TG \Rightarrow GT$
satisfying four axioms which say that $(G,\xi)\dd (\CA,T)\to (\CA,T)$ is a morphism, and 
$\delta \dd (G,\xi)\Rightarrow (G,\xi)(G,\xi)$ and $\varepsilon \dd 1\Rightarrow (G,\xi)$ are 2-cells, in  
$\mathrm{EM}(\mathfrak{K})$, and three axioms which say $\delta$ is coassociative with counit $\varepsilon$. 

Suppose $\mathfrak{K}$ admits the Eilenberg-Moore construction for both monads and comonads.
Simply because $\mathrm{em}\dd \mathrm{EM}(\mathfrak{K}) \lra \mathfrak{K}$ 
is a 2-functor, each mixed wreath $(G,\xi)\dd (\CA,T)\to (\CA,T)$ yields a comonad $G^{\xi}= \mathrm{em}(G,\xi)$ 
on $\CA^T$ in $\mathfrak{K}$. 
Define 
\begin{eqnarray*}
\mathrm{mem}(G,\xi,T) = (\CA^T)^{G^{\xi}} \ ,
\end{eqnarray*}
the Eilenberg-Moore construction for the comonad $G_{\xi}$.
This gives the object function for a {\em mixed Eilenberg-Moore construction}
\begin{eqnarray}\label{memfunctor}
\mathrm{mem} \dd \mathrm{EM}^{\mathrm{co}}(\mathrm{EM}(\KK))\stackrel{\mathrm{EM}^{\mathrm{co}}(\mathrm{em})}\lra \mathrm{EM}^{\mathrm{co}}(\KK)\stackrel{\mathrm{em}^{\mathrm{co}}}\lra \KK
\end{eqnarray}
with an obvious left adjoint.

We have the following description of $\mathrm{mem}(G,\xi,T)$
when $\KK= \mathrm{Cat}$. The objects $(A,a,c)$
of the category consist of an Eilenberg-Moore $T$-algebra
$a\dd TA\to A$ and a morphism $c\dd A\to GA$
satisfying the following three conditions.
\begin{eqnarray*}
\begin{aligned}
\xymatrix{
TA \ar[r]^-{a} \ar[d]_-{Tc} & A \ar[dd]^-{c} \\
TGA \ar[d]_-{\xi_A} & \\
GTA\ar[r]_-{Ga} & GA}
\qquad
\xymatrix{
A \ar[r]^-{c} \ar[d]_-{c} & GA \ar[dd]^-{Gc} \\
GA \ar[d]_-{\delta_A} & \\
GGTA\ar[r]_-{GGa} & GGA}
\qquad
\xymatrix{
A \ar[rdd]^-{1} \ar[d]_-{c} &  \\
GA \ar[d]_-{\varepsilon_A} & \\
GTA\ar[r]_-{a} & A}
\end{aligned}
\end{eqnarray*}

To reinforce the limit nature of the $\mathrm{mem}(G,\xi,T)$ construction
we next record its characterization as a representing object.
This can be taken as the definition when $\KK$ lacks the Eilenberg-Moore 
construction for monads or comonads in general.  

\begin{proposition}\label{upemgt}
$\KK(\CX,\mathrm{mem}(G,\xi,T))\cong \mathrm{mem}(\KK(\CX,G),\KK(\CX,\xi),\KK(\CX,T))$
\end{proposition}
  
\begin{definition}\label{defmopwr}
{\em Let $T$ be a monoid on $\CA$ in the 2-category $\mathfrak{K}$.
A {\em mixed opwreath around the monad} $T$ is a comonad on $(\CA,T)$ 
in the 2-category $\mathrm{KL}(\mathfrak{K})$. This consists of an endomorphism
$G$ of $\CA$ made into a morphism of $\mathrm{KL}(\mathfrak{K})$ by a
2-cell $\zeta\dd GT\Rightarrow TG$ and into a comonad by comultiplication
$\delta \dd G\Rightarrow TGG$ and counit $\varepsilon \dd G\Rightarrow T$.
The seven axioms are shown below in string form.}  
\end{definition}
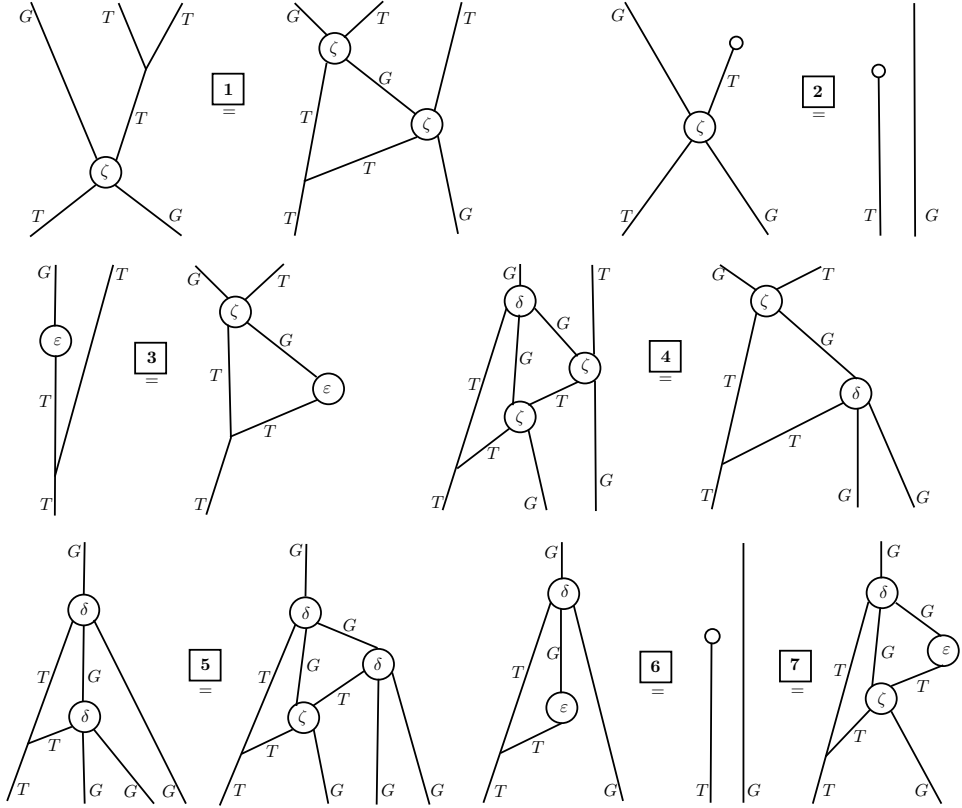
\begin{center}
\psscalebox{0.6 0.6} 
{
\begin{pspicture}(0,-2.6173604)(20.13,2.6173604)
\pscircle[linecolor=black, linewidth=0.04, dimen=outer](1.9,-1.1772718){0.36}
\pscircle[linecolor=black, linewidth=0.04, dimen=outer](8.94,-0.11727184){0.36}
\pscircle[linecolor=black, linewidth=0.04, dimen=outer](6.92,1.5627282){0.36}
\pscircle[linecolor=black, linewidth=0.04, dimen=outer](14.92,-0.17727184){0.36}
\rput[bl](1.78,-1.3372718){$\zeta$}
\rput[bl](8.84,-0.29727185){$\zeta$}
\rput[bl](6.8,1.3827281){$\zeta$}
\rput[bl](14.84,-0.35727185){$\zeta$}
\psline[linecolor=black, linewidth=0.04](0.26,2.5827281)(1.7,-0.87727183)
\psline[linecolor=black, linewidth=0.04](2.18,2.5627282)(2.78,1.1027281)(3.58,2.5627282)
\psline[linecolor=black, linewidth=0.04](2.78,1.1027281)(2.12,-0.93727183)
\psline[linecolor=black, linewidth=0.04](1.68,-1.4572718)(0.24,-2.597272)(0.24,-2.597272)
\psline[linecolor=black, linewidth=0.04](2.1,-1.4572718)(3.56,-2.577272)
\psline[linecolor=black, linewidth=0.04](6.04,2.5627282)(6.76,1.8427281)
\psline[linecolor=black, linewidth=0.04](7.98,2.6027281)(7.12,1.8027282)
\psline[linecolor=black, linewidth=0.04](7.16,1.3227282)(8.7,0.10272816)
\psline[linecolor=black, linewidth=0.04](6.74,1.2427281)(6.04,-2.577272)
\psline[linecolor=black, linewidth=0.04](9.18,-0.37727183)(9.62,-2.5372717)
\psline[linecolor=black, linewidth=0.04](8.72,-0.39727184)(6.24,-1.3772719)
\rput[bl](4.46,0.06272816){$=$}
\rput[bl](17.38,0.002728157){$=$}
\psframe[linecolor=black, linewidth=0.04, dimen=outer](4.94,1.0227282)(4.22,0.30272815)
\psframe[linecolor=black, linewidth=0.04, dimen=outer](17.88,0.96272814)(17.16,0.24272816)
\rput[bl](4.48,0.5427282){$\mathbf{1}$}
\rput[bl](17.4,0.48272815){$\mathbf{2}$}
\psline[linecolor=black, linewidth=0.04](13.28,2.5827281)(14.72,0.082728155)
\psline[linecolor=black, linewidth=0.04](14.74,-0.47727185)(13.22,-2.577272)
\psline[linecolor=black, linewidth=0.04](15.04,-0.47727185)(16.42,-2.557272)
\psline[linecolor=black, linewidth=0.04](15.1,0.10272816)(15.66,1.5627282)
\pscircle[linecolor=black, linewidth=0.04, dimen=outer](15.72,1.6827282){0.16}
\psline[linecolor=black, linewidth=0.04](18.88,-2.577272)(18.84,0.9227282)
\pscircle[linecolor=black, linewidth=0.04, dimen=outer](18.84,1.0627282){0.16}
\psline[linecolor=black, linewidth=0.04](19.62,2.5627282)(19.64,-2.597272)
\rput[bl](0.0,2.142728){$G$}
\rput[bl](1.82,2.122728){$T$}
\rput[bl](5.86,2.142728){$G$}
\rput[bl](7.88,0.78272814){$G$}
\rput[bl](12.98,2.162728){$G$}
\rput[bl](16.34,-2.2572718){$G$}
\rput[bl](19.86,-2.2572718){$G$}
\rput[bl](9.64,-2.2372718){$G$}
\rput[bl](3.28,-2.2572718){$G$}
\rput[bl](3.56,2.0827281){$T$}
\rput[bl](2.54,-0.11727184){$T$}
\rput[bl](6.16,-0.097271845){$T$}
\rput[bl](15.52,0.70272815){$T$}
\rput[bl](13.18,-2.2772717){$T$}
\rput[bl](7.54,-1.2172718){$T$}
\rput[bl](5.8,-2.317272){$T$}
\rput[bl](0.28,-2.2772717){$T$}
\rput[bl](7.84,2.1027281){$T$}
\rput[bl](9.74,2.1027281){$T$}
\psline[linecolor=black, linewidth=0.04](9.1,0.16272816)(9.7,2.5827281)
\rput[bl](18.52,-2.2572718){$T$}
\end{pspicture}
}
\bigskip

\psscalebox{0.6 0.6} 
{
\begin{pspicture}(0,-2.7873774)(19.53,2.7873774)
\pscircle[linecolor=black, linewidth=0.04, dimen=outer](4.34,1.7127452){0.36}
\pscircle[linecolor=black, linewidth=0.04, dimen=outer](15.98,1.9527452){0.36}
\pscircle[linecolor=black, linewidth=0.04, dimen=outer](0.4,1.0727452){0.36}
\pscircle[linecolor=black, linewidth=0.04, dimen=outer](10.58,-0.60725474){0.36}
\rput[bl](4.22,1.5527452){$\zeta$}
\rput[bl](15.84,1.7727453){$\zeta$}
\rput[bl](11.92,0.35274523){$\zeta$}
\rput[bl](10.48,-0.78725475){$\zeta$}
\psline[linecolor=black, linewidth=0.04](0.4,0.7527453)(0.38,-2.7872548)
\psline[linecolor=black, linewidth=0.04](0.4,2.7527452)(0.38,1.3927453)
\psline[linecolor=black, linewidth=0.04](3.46,2.7327452)(4.18,2.0127451)
\psline[linecolor=black, linewidth=0.04](5.4,2.7727451)(4.54,1.9727453)
\psline[linecolor=black, linewidth=0.04](4.58,1.4927453)(6.12,0.27274525)
\psline[linecolor=black, linewidth=0.04](1.66,2.7527452)(0.38,-1.9272548)
\psline[linecolor=black, linewidth=0.04](6.14,-0.22725475)(4.22,-1.0472548)
\rput[bl](2.36,0.092745245){$=$}
\rput[bl](13.62,0.13274525){$=$}
\psframe[linecolor=black, linewidth=0.04, dimen=outer](2.84,1.0527452)(2.12,0.33274525)
\psframe[linecolor=black, linewidth=0.04, dimen=outer](14.12,1.0927453)(13.4,0.37274525)
\psline[linecolor=black, linewidth=0.04](12.16,2.7527452)(12.2,0.77274525)
\psline[linecolor=black, linewidth=0.04](10.56,1.6527452)(10.42,-0.34725475)
\psline[linecolor=black, linewidth=0.04](12.22,0.19274525)(12.24,-2.6872547)
\psline[linecolor=black, linewidth=0.04](10.76,-0.34725475)(11.86,0.17274524)
\psline[linecolor=black, linewidth=0.04](17.98,-2.6072547)(17.98,-0.40725476)
\rput[bl](0.0,2.4527452){$G$}
\rput[bl](3.28,2.3127453){$G$}
\rput[bl](5.3,0.95274526){$G$}
\rput[bl](10.2,2.4327452){$G$}
\rput[bl](11.38,1.3327452){$G$}
\rput[bl](10.56,0.57274526){$G$}
\rput[bl](12.32,-2.1272547){$G$}
\rput[bl](0.06,-2.6672547){$T$}
\rput[bl](0.04,-0.38725474){$T$}
\rput[bl](12.28,2.4127452){$T$}
\rput[bl](11.36,-0.38725474){$T$}
\rput[bl](4.96,-1.0472548){$T$}
\rput[bl](3.44,-2.6672547){$T$}
\rput[bl](1.72,2.4127452){$T$}
\rput[bl](5.26,2.2727451){$T$}
\rput[bl](3.78,0.19274525){$T$}
\rput[bl](14.54,-2.4672546){$T$}
\pscircle[linecolor=black, linewidth=0.04, dimen=outer](6.38,0.012745247){0.36}
\pscircle[linecolor=black, linewidth=0.04, dimen=outer](17.94,-0.087254755){0.36}
\pscircle[linecolor=black, linewidth=0.04, dimen=outer](10.58,1.9527452){0.36}
\pscircle[linecolor=black, linewidth=0.04, dimen=outer](12.0,0.47274524){0.36}
\rput[bl](0.36,0.9927452){$\varepsilon$}
\rput[bl](6.26,-0.06725475){$\varepsilon$}
\rput[bl](2.4,0.59274524){$\mathbf{3}$}
\rput[bl](13.66,0.6127452){$\mathbf{4}$}
\psline[linecolor=black, linewidth=0.04](4.178868,1.4327452)(4.24,-1.071101)(3.7,-2.7672548)
\rput[bl](10.48,1.7927452){$\delta$}
\rput[bl](17.86,-0.20725475){$\delta$}
\psline[linecolor=black, linewidth=0.04](10.58,2.2927454)(10.58,2.7727451)
\psline[linecolor=black, linewidth=0.04](10.9,1.7927452)(11.84,0.73274523)
\psline[linecolor=black, linewidth=0.04](9.2,-1.7472547)(10.34,-0.86725473)
\psline[linecolor=black, linewidth=0.04](10.76,-0.8872548)(11.16,-2.6672547)
\psline[linecolor=black, linewidth=0.04](10.28,1.7927452)(8.88,-2.6672547)
\rput[bl](9.86,-1.5472548){$T$}
\rput[bl](9.44,-0.0072547533){$T$}
\rput[bl](8.64,-2.6272547){$T$}
\psline[linecolor=black, linewidth=0.04](14.96,2.7527452)(15.76,2.1927452)
\psline[linecolor=black, linewidth=0.04](17.18,2.7127452)(16.2,2.2127452)
\psline[linecolor=black, linewidth=0.04](16.24,1.7527453)(17.96,0.23274525)
\psline[linecolor=black, linewidth=0.04](15.76,1.6927452)(14.78,-2.6472547)
\psline[linecolor=black, linewidth=0.04](17.68,-0.28725475)(15.02,-1.6472547)
\psline[linecolor=black, linewidth=0.04](18.22,-0.28725475)(19.22,-2.6072547)
\rput[bl](17.16,0.97274524){$G$}
\rput[bl](19.26,-2.5072548){$G$}
\rput[bl](10.66,-2.6472547){$G$}
\rput[bl](17.54,-2.4672546){$G$}
\rput[bl](14.78,2.4127452){$G$}
\rput[bl](17.2,2.4127452){$T$}
\rput[bl](15.04,0.052745245){$T$}
\rput[bl](16.46,-1.2672547){$T$}
\end{pspicture}
}

\bigskip

\psscalebox{0.6 0.6} 
{
\begin{pspicture}(0,-2.923827)(20.89881,2.923827)
\pscircle[linecolor=black, linewidth=0.04, dimen=outer](1.6988101,1.403827){0.36}
\pscircle[linecolor=black, linewidth=0.04, dimen=outer](20.53881,0.503827){0.36}
\rput[bl](19.09881,-0.73617303){$\zeta$}
\rput[bl](6.39881,-1.136173){$\zeta$}
\psline[linecolor=black, linewidth=0.04](15.45881,0.683827)(15.43881,-2.856173)
\psline[linecolor=black, linewidth=0.04](1.6988101,1.063827)(1.6788101,-0.616173)
\psline[linecolor=black, linewidth=0.04](1.9188101,-1.236173)(3.23881,-2.876173)
\psline[linecolor=black, linewidth=0.04](1.9188101,1.1838269)(3.93881,-2.876173)
\rput[bl](17.19881,-0.476173){$=$}
\rput[bl](4.21881,-0.476173){$=$}
\psframe[linecolor=black, linewidth=0.04, dimen=outer](17.65881,0.52382696)(16.93881,-0.19617301)
\psline[linecolor=black, linewidth=0.04](1.6788101,-1.296173)(1.7188101,-2.876173)
\psline[linecolor=black, linewidth=0.04](16.15881,2.883827)(16.15881,-2.856173)
\rput[bl](1.3388101,2.583827){$G$}
\rput[bl](6.19881,2.603827){$G$}
\rput[bl](11.79881,2.563827){$G$}
\rput[bl](1.8188101,-2.716173){$G$}
\rput[bl](2.5788102,-2.736173){$G$}
\rput[bl](3.4188101,-2.736173){$G$}
\rput[bl](19.17881,0.34382698){$G$}
\rput[bl](15.11881,-2.696173){$T$}
\rput[bl](0.9188101,-1.716173){$T$}
\rput[bl](4.9788103,-2.716173){$T$}
\rput[bl](10.698811,-2.736173){$T$}
\rput[bl](11.51881,-1.756173){$T$}
\rput[bl](0.6788101,-0.21617302){$T$}
\rput[bl](0.23881012,-2.676173){$T$}
\rput[bl](11.07881,-0.09617302){$T$}
\rput[bl](18.57881,-1.516173){$T$}
\pscircle[linecolor=black, linewidth=0.04, dimen=outer](19.19881,1.783827){0.36}
\pscircle[linecolor=black, linewidth=0.04, dimen=outer](12.21881,1.763827){0.36}
\pscircle[linecolor=black, linewidth=0.04, dimen=outer](12.17881,-0.756173){0.36}
\rput[bl](12.13881,-0.85617304){$\varepsilon$}
\rput[bl](1.6188102,1.283827){$\delta$}
\rput[bl](18.13881,-0.05617302){$T$}
\rput[bl](17.89881,-2.716173){$T$}
\rput[bl](19.95881,-0.416173){$T$}
\rput[bl](19.91881,-2.696173){$G$}
\rput[bl](11.83881,0.32382697){$G$}
\rput[bl](20.03881,1.243827){$G$}
\rput[bl](13.09881,-2.696173){$G$}
\rput[bl](16.23881,-2.696173){$G$}
\rput[bl](5.5188103,-0.21617302){$T$}
\rput[bl](5.69881,-1.836173){$T$}
\rput[bl](7.25881,-0.696173){$T$}
\pscircle[linecolor=black, linewidth=0.04, dimen=outer](1.7188101,-0.956173){0.36}
\pscircle[linecolor=black, linewidth=0.04, dimen=outer](6.55881,1.343827){0.36}
\pscircle[linecolor=black, linewidth=0.04, dimen=outer](8.15881,0.20382698){0.36}
\pscircle[linecolor=black, linewidth=0.04, dimen=outer](6.5188103,-0.97617304){0.36}
\psline[linecolor=black, linewidth=0.04](1.7188101,2.903827)(1.6988101,1.743827)
\psline[linecolor=black, linewidth=0.04](6.57881,2.863827)(6.55881,1.703827)
\psline[linecolor=black, linewidth=0.04](1.4588101,1.163827)(0.01881012,-2.816173)
\psline[linecolor=black, linewidth=0.04](1.4388101,-1.176173)(0.45881012,-1.556173)
\psframe[linecolor=black, linewidth=0.04, dimen=outer](4.71881,0.543827)(3.99881,-0.17617302)
\psline[linecolor=black, linewidth=0.04](6.33881,1.083827)(4.67881,-2.916173)
\psline[linecolor=black, linewidth=0.04](6.27881,-1.236173)(5.15881,-1.776173)
\psline[linecolor=black, linewidth=0.04](6.57881,1.003827)(6.35881,-0.716173)
\psline[linecolor=black, linewidth=0.04](7.85881,0.06382698)(6.71881,-0.716173)
\psline[linecolor=black, linewidth=0.04](6.79881,1.123827)(8.13881,0.563827)
\psline[linecolor=black, linewidth=0.04](6.73881,-1.236173)(7.05881,-2.876173)
\psline[linecolor=black, linewidth=0.04](8.15881,-0.11617302)(8.11881,-2.896173)
\psline[linecolor=black, linewidth=0.04](8.45881,0.023826981)(9.2788105,-2.896173)
\rput[bl](7.09881,-2.696173){$G$}
\rput[bl](8.21881,-2.696173){$G$}
\rput[bl](9.29881,-2.696173){$G$}
\rput[bl](6.57881,0.06382698){$G$}
\rput[bl](7.37881,0.943827){$G$}
\rput[bl](1.6388102,-1.0761731){$\delta$}
\rput[bl](8.05881,0.08382698){$\delta$}
\rput[bl](6.4788103,1.203827){$\delta$}
\rput[bl](1.7788101,-0.19617301){$G$}
\rput[bl](4.25881,0.08382698){$\mathbf{5}$}
\pscircle[linecolor=black, linewidth=0.04, dimen=outer](15.47881,0.82382697){0.18}
\rput[bl](14.13881,-0.49617302){$=$}
\psframe[linecolor=black, linewidth=0.04, dimen=outer](14.59881,0.503827)(13.87881,-0.21617302)
\pscircle[linecolor=black, linewidth=0.04, dimen=outer](19.17881,-0.556173){0.36}
\psline[linecolor=black, linewidth=0.04](12.15881,1.4238269)(12.15881,-0.43617302)
\psline[linecolor=black, linewidth=0.04](11.93881,1.583827)(10.45881,-2.856173)
\psline[linecolor=black, linewidth=0.04](12.17881,-1.096173)(10.83881,-1.756173)
\psline[linecolor=black, linewidth=0.04](12.43881,1.503827)(13.51881,-2.816173)
\psline[linecolor=black, linewidth=0.04](12.17881,2.903827)(12.17881,2.083827)
\rput[bl](12.15881,1.6838269){$\delta$}
\rput[bl](19.13881,1.663827){$\delta$}
\rput[bl](20.51881,0.423827){$\varepsilon$}
\psline[linecolor=black, linewidth=0.04](19.17881,2.923827)(19.17881,2.103827)
\psline[linecolor=black, linewidth=0.04](19.15881,1.463827)(18.97881,-0.296173)
\psline[linecolor=black, linewidth=0.04](19.39881,-0.27617303)(20.53881,0.18382698)
\psline[linecolor=black, linewidth=0.04](20.49881,0.82382697)(19.49881,1.563827)
\psline[linecolor=black, linewidth=0.04](18.89881,1.583827)(17.67881,-2.876173)
\psline[linecolor=black, linewidth=0.04](18.93881,-0.79617304)(17.95881,-1.836173)
\psline[linecolor=black, linewidth=0.04](19.39881,-0.836173)(20.49881,-2.876173)
\rput[bl](19.298811,2.583827){$G$}
\rput[bl](14.11881,0.043826982){$\mathbf{6}$}
\rput[bl](17.17881,0.08382698){$\mathbf{7}$}
\end{pspicture}
}

\end{center}

At the 2-category level Definition~\ref{defmopwr} really is just an example: 
a mixed opwreath in $\KK$ is a mixed wreath in $\KK^{\mathrm{op}}$.
Indeed, in the presence of right adjoints, we will now point out how 
a mixed opwreath amounts to a wreath.  

Recall from \cite{7} the terminology and concept of mates under adjunction.
Here is an exercise on mates using the string calculus.

\begin{proposition}\label{wrmates}
Suppose $T$ is a monad on $\CA\in \mathfrak{K}$.
Suppose $G\dashv S$ are adjoint endomorphisms of $\CA$.
Mixed opwreath structures on $G$ around $T$ correspond under adjoint
mateship to wreath structures on $S$ around $T$.
\end{proposition}

In the situation of Proposition~\ref{wrmates}, the Eilenberg-Moore
construction for the wreath product $ST$ does not easily reinterpret 
in terms of $G$ and $T$, rather, as you would expect, the Kleisli 
construction does. We shall now define this in general.

By applying Proposition~\ref{upemgt} to $\KK^{\mathrm{op}}$, 
and by defining the {\em mixed Kleisli construction} as the composite
\begin{eqnarray}
\mathrm{mkl} \dd \mathrm{KL}^{\mathrm{co}}(\mathrm{KL}(\KK))\stackrel{\mathrm{KL}^{\mathrm{co}}(\mathrm{kl})}\lra \mathrm{KL}^{\mathrm{co}}(\KK)\stackrel{\mathrm{kl}^{\mathrm{co}}}\lra \KK \ ,
\end{eqnarray}
we obtain:
\begin{proposition}\label{upklgt}
$\KK(\mathrm{mkl}(G,\zeta,T),\CX)\cong \mathrm{mem}(\KK(G,\CX),\KK(\zeta,\CX),\KK(T,\CX))$
\end{proposition}

For $\KK=\mathrm{Cat}$, the category $\mathrm{mkl}(G,\zeta,T)$ has the same objects as
$\CA$ and has homsets defined by
\begin{eqnarray*}
\mathrm{mkl}(G,\zeta,T)(A,B)=\CA(GA,TB) \ .
\end{eqnarray*}
Composition is defined by {\em wreath convolution}: the composite of $f\dd GA\to TB$
and $g\dd GB\to TC$ is $g\circ f = f*_{\zeta}g$ as in the commutative diagram \eqref{wrconv}.
 \begin{eqnarray}\label{wrconv}
\begin{aligned}
\xymatrix{
GA \ar[r]^-{\delta_A} \ar[d]_-{f*_{\zeta}g} & TGGA\ar[r]^-{TGf} & TGTB \ar[d]^-{T\zeta_{B}} \\
TC  & TTTC \ar[l]^-{\mu_{3 C}} & TTGB \ar[l]^-{TTg} }
\end{aligned}
\end{eqnarray}
 \begin{center}
 \psscalebox{0.6 0.6} 
{
\begin{pspicture}(0,-4.8833404)(6.46,4.8833404)
\pscircle[linecolor=black, linewidth=0.04, dimen=outer](4.58,3.4033403){0.34}
\pscircle[linecolor=black, linewidth=0.04, dimen=outer](5.82,-2.1566596){0.34}
\pscircle[linecolor=black, linewidth=0.04, dimen=outer](4.46,-0.3966597){0.34}
\pscircle[linecolor=black, linewidth=0.04, dimen=outer](5.8,1.4233403){0.34}
\psline[linecolor=black, linewidth=0.04](4.58,4.8833404)(4.58,3.7233403)
\psline[linecolor=black, linewidth=0.04](4.8,3.1833403)(5.64,1.6833403)
\psline[linecolor=black, linewidth=0.04](4.58,3.1033404)(4.26,-0.15665969)
\psline[linecolor=black, linewidth=0.04](5.6,1.1633403)(4.68,-0.1766597)
\psline[linecolor=black, linewidth=0.04](6.06,1.2233403)(6.06,-1.9566597)
\psline[linecolor=black, linewidth=0.04](4.66,-0.6366597)(5.56,-1.9366597)
\psline[linecolor=black, linewidth=0.04](4.3,3.2433403)(1.42,-4.87666)
\psline[linecolor=black, linewidth=0.04](4.24,-0.5966597)(1.78,-3.8966596)(5.6,-2.3966596)
\psline[linecolor=black, linewidth=0.04](6.02,-2.4166596)(6.02,-4.87666)
\psline[linecolor=black, linewidth=0.04](6.04,1.6433403)(6.0,4.8433404)
\rput[bl](0.0,0.30334032){\Large{$f*_{\zeta}g=$}}
\rput[bl](5.68,1.2433403){$f$}
\rput[bl](5.74,-2.2966597){$g$}
\rput[bl](4.36,-0.5566597){$\zeta$}
\rput[bl](6.1,3.3233404){$A$}
\rput[bl](6.18,-0.4566597){$B$}
\rput[bl](6.18,-3.8766596){$C$}
\rput[bl](4.5,3.2633402){$\delta$}
\rput[bl](2.84,-0.1166597){$T$}
\rput[bl](5.26,2.4633403){$G$}
\rput[bl](1.26,-4.5366597){$T$}
\rput[bl](3.32,-2.1966598){$T$}
\rput[bl](4.0,-3.3966596){$T$}
\rput[bl](4.78,-1.5966597){$G$}
\rput[bl](4.52,1.5433403){$G$}
\rput[bl](5.24,0.2233403){$T$}
\rput[bl](4.2,4.4233403){$G$}
\end{pspicture}
}
 \end{center}
 
 \begin{remark}
 {\em 
 \begin{enumerate}
 \item For any monoidal category $\CV$, the functor $\CV\to [\CV,\CV]$, 
 taking $X$ to $X\otimes -$, is strong monoidal. 
So each monoid $A=(A,m,j)$ in $\CV$ is taken to a monad 
$T=(T,\mu,\eta)$ on $\CV$.
We will speak of a {\em mixed opwreath} around $A$ to mean a quadruple
$(C,d,w,z)$ consisting of an object $C$ and morphisms $d\dd C\to A\otimes C\otimes C$,
$w\dd C\to A$ and $z\dd C\otimes A\to A\otimes C$ satisfying the string diagram conditions
of Definition~\ref{defmopwr} with $T,\mu,\eta,G,\delta,\varepsilon,\zeta$ replaced by
$A,m,j,C,d,w,z$, respectively.  
Moreover, since mixed opwreaths are defined purely in terms of the monoidal
structure, each mixed opwreath $(C,d,w,z)$ around $A$ in $\CV$ defines gives
rise to a mixed opwreath $(G,\delta,\varepsilon,\zeta)$ around the monad $T=A\otimes -$.
Furthermore, we write $\mathrm{mkl}(C,z,A)$ for the category $\mathrm{mkl}(G,\zeta,T)$.
\item For any monoidal category $\CV$, the functor $\CV^{\mathrm{rev}}\to [\CV,\CV]$, 
 taking $X$ to $-\otimes X$, is also strong monoidal.
Thus the mixed opwreath around $A$ as in Item~1 is taken to a mixed wreath around
the monad $-\otimes A$ on $\CV$. 
We write $\mathrm{mem}(C,z,A)$ for the mixed Eilenberg-Moore construction \eqref{memfunctor} applied to this mixed wreath. 
\end{enumerate}}
 \end{remark}
 
 \begin{definition}
{\em The set $\mathrm{mkl}(C,z,A)(I,I) \cong \CV(C,A)$ of endomorphisms of $I$ in the category
$\mathrm{mkl}(C,z,A)$ is, of course, a monoid under composition. 
The multiplication might be called
{\em $z$-parametrized convolution} on $\CV(C,A)$.}  
\end{definition}
 
 \begin{example}[The Heisenberg category]\label{ExHsbg}
 {\em 
Suppose $A=(A,m,j,c,e)$ is a bimonoid in the braided monoidal category $\CV$. 
Denote the braiding by $\sigma$.
We obtain a mixed opwreath on the monoid $A=(A,m,j)$ 
(indeed it is a mixed opdistributive law) in $\CV$ by taking the comonoid
$C$ to be $A=(A,c,e)$, $z$ to be
\begin{eqnarray*}
z_{\mathrm{h}}=\left(A\otimes A\stackrel{1_A\otimes \sigma^{-1}_{A,A}c}\lra A\otimes A\otimes A \stackrel{\sigma^{-1}_{A,A} \otimes 1_A}\lra A\otimes A\otimes A\stackrel{1_A \otimes m}\lra A\otimes A \right) \\
=  \left(A\otimes A\stackrel{1_A\otimes c}\lra A\otimes A\otimes A \stackrel{\sigma^{-1}_{A,A\otimes A}}\lra A\otimes A\otimes A\stackrel{1_A \otimes m}\lra A\otimes A \right) \ ,
\end{eqnarray*}
$d$ to be $j\otimes c\dd A\to A\otimes A\otimes A$, and $w$ to be $j\circ e\dd A\to A$.
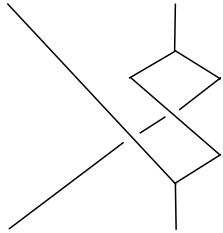
\begin{center}
\psscalebox{0.5 0.5} 
{
\begin{pspicture}(0,-3.0069635)(5.647003,3.0069635)
\psline[linecolor=black, linewidth=0.04](4.434674,2.9933739)(4.414674,1.753374)(5.594674,1.013374)(5.614674,0.973374)
\psline[linecolor=black, linewidth=0.04](4.434674,1.733374)(3.214674,1.013374)
\psline[linecolor=black, linewidth=0.04](4.454674,-3.0066261)(4.434674,-1.766626)(5.634674,-1.006626)
\psline[linecolor=black, linewidth=0.04](4.454674,-1.806626)(0.014673919,2.9933739)
\psline[linecolor=black, linewidth=0.04](3.254674,1.033374)(5.614674,-1.026626)
\psline[linecolor=black, linewidth=0.04](0.05467392,-2.986626)(3.074674,-0.686626)
\psline[linecolor=black, linewidth=0.04](3.494674,-0.446626)(4.174674,-0.0066260146)
\psline[linecolor=black, linewidth=0.04](4.5346737,0.17337398)(5.634674,1.033374)
\end{pspicture}
}
\end{center}
We put $\mathrm{Hb}(A) = \mathrm{mkl}(A,z_{\mathrm{h}},A)$ and call it 
the {\em Heisenberg category} of the bimonoid $A$ in $\CV$. Here is the reason.

\begin{proposition}
Suppose $\CV = \mathrm{Vect}$ is the symmetric monoidal category of vector spaces over a 
fixed field and $H$ is a Hopf algebra, then the $z_{\mathrm{h}}$-parametrized convolution of 
linear endomorphisms of $H$ is the Heisenberg product as defined in \cite{MoreiraPhD, AFM2015}. 
\end{proposition}
 }\end{example}

\section{Twisted coactions}\label{sectiontwcoact}

The construction of a mixed opwreath explained in this section is a dual of the wreath 
construction appearing as Example 3.3 in \cite{71} based on 
Sweedler's crossed product of Hopf algebras.

We begin by pointing out that, given a monoid $A$ in any monoidal category $\CV$,
the representable functor $\CV(-,A)\dd \CV^{\mathrm{op}}\to \mathrm{Set}$ becomes monoidal
when equipped with the natural family of functions
\begin{eqnarray}
\CV(X,A)\times \CV(Y,A)\lra \CV(X\otimes Y,A) \ ,
\end{eqnarray}
defined by $(u,v)\mapsto u\bullet v$ as depicted below, and $\eta_A\in \CV(I,A)$.
The reason is that the Yoneda embedding $\CV\to [\CV^{\mathrm{op}},\mathrm{Set}]$
is monoidal, where monoids in the codomain are precisely monoidal functors.

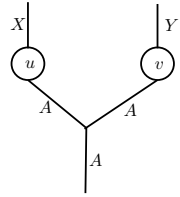
\begin{center}
\psscalebox{0.6 0.6} 
{
\begin{pspicture}(0,-2.1101382)(3.66,2.1101382)
\pscircle[linecolor=black, linewidth=0.04, dimen=outer](0.38,0.7501381){0.38}
\pscircle[linecolor=black, linewidth=0.04, dimen=outer](3.22,0.7501381){0.38}
\psline[linecolor=black, linewidth=0.04](0.38,1.0901381)(0.38,2.1101382)
\psline[linecolor=black, linewidth=0.04](3.22,1.0901381)(3.22,2.070138)
\psline[linecolor=black, linewidth=0.04](0.38,0.3901381)(1.66,-0.6698619)(1.64,-2.1098619)
\psline[linecolor=black, linewidth=0.04](3.22,0.3901381)(1.62,-0.6898619)
\rput[bl](0.32,0.6501381){$u$}
\rput[bl](3.16,0.6301381){$v$}
\rput[bl](0.0,1.490138){$X$}
\rput[bl](3.38,1.4701381){$Y$}
\rput[bl](0.62,-0.34986192){$A$}
\rput[bl](2.48,-0.40986192){$A$}
\rput[bl](1.72,-1.509862){$A$}
\end{pspicture}
}
\end{center}
Here are the properties of the dot product:
\begin{itemize}
\item[\em{(unitality)}] $\eta_A\bullet u=u=u\bullet \eta_A$
\item[\em{(associativity)}] $(u\bullet v)\bullet w = u\bullet (v\bullet w)$
\item[\em{(naturality)}] $(u\bullet v)\circ (f\otimes g) = (u\circ f)\bullet (v\circ g)$
\end{itemize}

We also recall that, if $\CV$ is (lax) braided then the tensor product $A\otimes B$ of monoids $A$ and $B$ is again a monoid: the (lax) braiding gives a distributive law of $A$ over $B$ used in defining $\mu_{A\otimes B}$
in terms of the multiplications $\mu_A$ and $\mu_B$.

\begin{definition}\label{twcoact}
{\em Let $A$ be a monoid and $B$ be a bimonoid in a (lax) braided monoidal category $\CV$. 
A {\em twisted (right) coaction} of $B$ on $A$ consists of a monoid morphism
$\gamma\dd A\to A\otimes B$ and a morphism $\tau\dd I\to A\otimes B^{\otimes 2}$
such that
\begin{itemize}
\item[\em{(counitality)}] 
$\left( 1_A\otimes \varepsilon_B \right)\circ \gamma = 1_A$
\item[\em{($\tau$-coassociativity)}] 
$\tau\bullet \left( (\gamma \otimes 1_B)\circ \gamma \right) = \left( (1_A\otimes \delta_B)\circ \gamma \right)\bullet \tau$
\item[\em{(2-cocyclicity)}] $\left( (1_A\otimes \delta_B\otimes 1_B)\circ \tau \right)\bullet (\tau\otimes \eta_B) 
= \left( (1_{A\otimes B}\otimes \delta_B)\circ \tau \right)\bullet \left( (\gamma \otimes 1_{B\otimes B})\circ  \tau)\right)$
\item[\em{(normality)}]
$(1_{A\otimes B}\otimes \varepsilon_B)\circ \tau = \eta_A\otimes\eta_B = (1_{A}\otimes \varepsilon_B\otimes 1_B)\circ \tau  \ .$
\end{itemize}
Note that we do not require $\tau$ to be $\bullet$-invertible.
}  
\end{definition}  
Here in string form are the conditions on a twisted coaction.
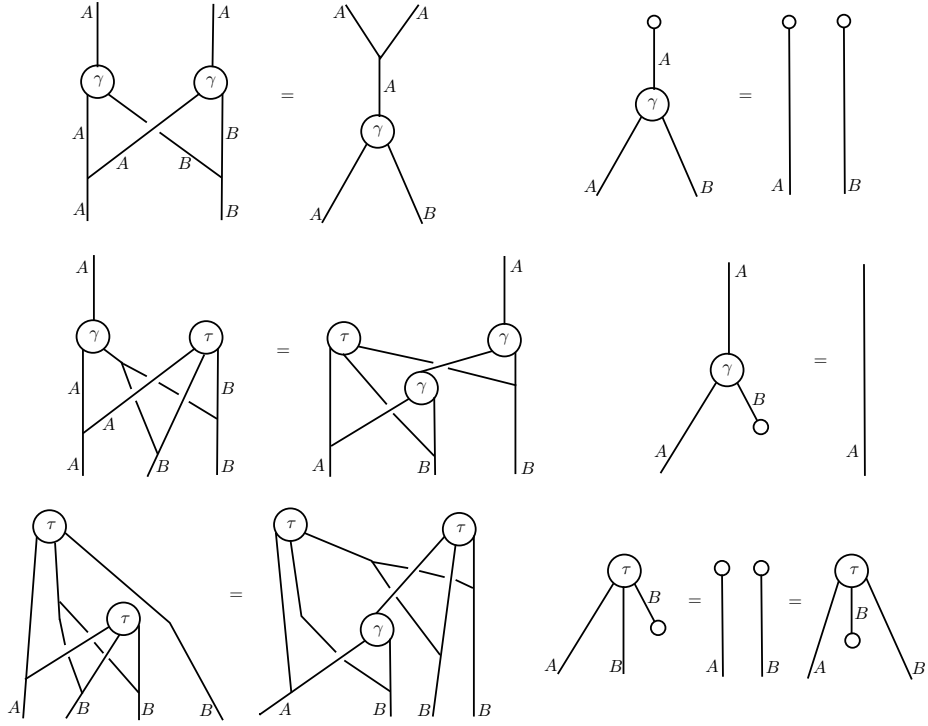
\begin{center}
\psscalebox{0.6 0.6} 
{
\begin{pspicture}(0,-2.4539683)(17.28,2.4539683)
\pscircle[linecolor=black, linewidth=0.04, dimen=outer](0.56,0.6939682){0.38}
\pscircle[linecolor=black, linewidth=0.04, dimen=outer](3.04,0.6739682){0.38}
\pscircle[linecolor=black, linewidth=0.04, dimen=outer](6.7,-0.4060318){0.38}
\rput[bl](4.58,0.2939682){$\Large{=}$}
\rput[bl](0.48,0.5939682){$\gamma$}
\rput[bl](0.18,2.0939682){$A$}
\rput[bl](7.68,-2.346032){$B$}
\rput[bl](2.94,0.5739682){$\gamma$}
\rput[bl](6.6,-0.5260318){$\gamma$}
\rput[bl](3.18,2.0939682){$A$}
\rput[bl](5.72,2.0939682){$A$}
\rput[bl](7.54,2.0739682){$A$}
\rput[bl](5.2,-2.366032){$A$}
\rput[bl](0.02,-2.2460318){$A$}
\rput[bl](0.96,-1.2260318){$A$}
\rput[bl](17.0,-1.7260318){$B$}
\rput[bl](2.32,-1.2260318){$B$}
\rput[bl](3.36,-0.5660318){$B$}
\rput[bl](3.36,-2.3060317){$B$}
\psline[linecolor=black, linewidth=0.04](0.56,2.4539683)(0.56,1.0139682)
\psline[linecolor=black, linewidth=0.04](3.1,2.413968)(3.08,1.0139682)
\psline[linecolor=black, linewidth=0.04](0.34,0.4139682)(0.34,-2.3860319)
\psline[linecolor=black, linewidth=0.04](3.3,0.4339682)(3.28,-2.366032)
\psline[linecolor=black, linewidth=0.04](2.8,0.4339682)(0.34,-1.4460318)
\psline[linecolor=black, linewidth=0.04](0.8,0.4139682)(1.64,-0.2060318)
\psline[linecolor=black, linewidth=0.04](1.92,-0.4260318)(3.28,-1.4260318)
\psline[linecolor=black, linewidth=0.04](6.74,1.2739682)(6.74,-0.0860318)
\psline[linecolor=black, linewidth=0.04](6.46,-0.6860318)(5.48,-2.4260318)
\psline[linecolor=black, linewidth=0.04](6.92,-0.6860318)(7.68,-2.4460318)
\psline[linecolor=black, linewidth=0.04](6.0,2.393968)(6.76,1.2139682)(7.6,2.393968)
\rput[bl](6.82,0.4739682){$A$}
\rput[bl](0.0,-0.5660318){$A$}
\pscircle[linecolor=black, linewidth=0.04, dimen=outer](12.72,0.1939682){0.38}
\rput[bl](13.76,-1.7660317){$B$}
\rput[bl](12.62,0.0739682){$\gamma$}
\psline[linecolor=black, linewidth=0.04](12.76,1.8739682)(12.76,0.5139682)
\psline[linecolor=black, linewidth=0.04](12.48,-0.0860318)(11.5,-1.8260318)
\psline[linecolor=black, linewidth=0.04](12.94,-0.0860318)(13.7,-1.8460318)
\rput[bl](12.84,1.0739682){$A$}
\rput[bl](14.62,0.2939682){$\Large{=}$}
\pscircle[linecolor=black, linewidth=0.04, dimen=outer](15.72,2.0139682){0.16}
\pscircle[linecolor=black, linewidth=0.04, dimen=outer](12.76,2.0139682){0.16}
\psline[linecolor=black, linewidth=0.04](15.72,1.8739682)(15.74,-1.8060318)
\pscircle[linecolor=black, linewidth=0.04, dimen=outer](16.92,2.0339682){0.16}
\psline[linecolor=black, linewidth=0.04](16.92,1.8939682)(16.94,-1.7860318)
\rput[bl](11.22,-1.7460318){$A$}
\rput[bl](15.32,-1.7260318){$A$}
\end{pspicture}
}
\end{center}
\begin{center}
\psscalebox{0.6 0.6} 
{
\begin{pspicture}(0,-2.45)(17.52,2.45)
\pscircle[linecolor=black, linewidth=0.04, dimen=outer](0.58,0.65){0.38}
\pscircle[linecolor=black, linewidth=0.04, dimen=outer](3.06,0.63){0.38}
\pscircle[linecolor=black, linewidth=0.04, dimen=outer](7.78,-0.49){0.38}
\rput[bl](4.6,0.25){$\Large{=}$}
\rput[bl](0.5,0.55){$\gamma$}
\rput[bl](0.2,2.05){$A$}
\rput[bl](9.96,-2.39){$B$}
\rput[bl](7.68,-0.61){$\gamma$}
\rput[bl](5.38,-2.33){$A$}
\rput[bl](0.0,-2.31){$A$}
\rput[bl](0.78,-1.41){$A$}
\rput[bl](1.96,-2.35){$B$}
\rput[bl](3.38,-0.61){$B$}
\rput[bl](3.38,-2.35){$B$}
\psline[linecolor=black, linewidth=0.04](0.6,2.45)(0.6,1.01)
\psline[linecolor=black, linewidth=0.04](0.36,0.37)(0.36,-2.43)
\psline[linecolor=black, linewidth=0.04](3.32,0.39)(3.3,-2.41)
\psline[linecolor=black, linewidth=0.04](2.82,0.39)(0.36,-1.49)
\psline[linecolor=black, linewidth=0.04](0.82,0.37)(1.24,0.05)
\psline[linecolor=black, linewidth=0.04](2.68,-0.79)(3.3,-1.17)
\rput[bl](9.72,2.07){$A$}
\rput[bl](0.02,-0.61){$A$}
\pscircle[linecolor=black, linewidth=0.04, dimen=outer](14.48,-0.09){0.38}
\rput[bl](15.02,-0.87){$B$}
\rput[bl](14.38,-0.21){$\gamma$}
\psline[linecolor=black, linewidth=0.04](14.52,2.29)(14.52,0.25)
\psline[linecolor=black, linewidth=0.04](14.24,-0.37)(13.02,-2.37)
\psline[linecolor=black, linewidth=0.04](14.7,-0.37)(15.14,-1.21)
\rput[bl](14.64,1.97){$A$}
\rput[bl](16.38,0.01){$\Large{=}$}
\psline[linecolor=black, linewidth=0.04](17.48,2.25)(17.5,-2.43)
\rput[bl](12.88,-2.01){$A$}
\rput[bl](17.08,-2.01){$A$}
\rput[bl](3.02,0.55){$\tau$}
\rput[bl](6.0,0.55){$\tau$}
\psline[linecolor=black, linewidth=0.04](3.02,0.25)(1.78,-2.43)(1.78,-2.45)
\psline[linecolor=black, linewidth=0.04](1.2,0.09)(1.42,-0.51)
\psline[linecolor=black, linewidth=0.04](1.52,-0.77)(2.0,-1.95)
\psline[linecolor=black, linewidth=0.04](1.24,0.05)(1.76,-0.23)
\psline[linecolor=black, linewidth=0.04](2.02,-0.37)(2.46,-0.63)
\pscircle[linecolor=black, linewidth=0.04, dimen=outer](6.08,0.61){0.38}
\pscircle[linecolor=black, linewidth=0.04, dimen=outer](9.6,0.57){0.38}
\rput[bl](9.48,0.47){$\gamma$}
\psline[linecolor=black, linewidth=0.04](9.6,2.43)(9.6,0.91)
\psline[linecolor=black, linewidth=0.04](9.34,0.33)(7.76,-0.13)
\psline[linecolor=black, linewidth=0.04](9.84,0.31)(9.84,-2.39)
\psline[linecolor=black, linewidth=0.04](8.06,-0.69)(8.08,-2.41)
\psline[linecolor=black, linewidth=0.04](5.78,0.41)(5.78,-2.45)
\psline[linecolor=black, linewidth=0.04](7.5,-0.73)(5.8,-1.77)
\psline[linecolor=black, linewidth=0.04](6.06,0.27)(7.02,-0.85)
\psline[linecolor=black, linewidth=0.04](7.16,-1.05)(8.06,-1.99)
\psline[linecolor=black, linewidth=0.04](6.4,0.43)(8.04,0.07)
\psline[linecolor=black, linewidth=0.04](8.36,-0.05)(9.86,-0.43)
\rput[bl](7.68,-2.37){$B$}
\pscircle[linecolor=black, linewidth=0.04, dimen=outer](15.22,-1.35){0.18}
\end{pspicture}
}
\end{center}

\begin{center}
\psscalebox{0.6 0.6} 
{
\begin{pspicture}(0,-2.354768)(20.1,2.354768)
\pscircle[linecolor=black, linewidth=0.04, dimen=outer](0.92,1.9347681){0.38}
\pscircle[linecolor=black, linewidth=0.04, dimen=outer](2.54,-0.105231896){0.38}
\pscircle[linecolor=black, linewidth=0.04, dimen=outer](8.1,-0.3452319){0.38}
\rput[bl](4.86,0.33476812){$\Large{=}$}
\rput[bl](10.3,-2.2452319){$B$}
\rput[bl](0.0,-2.185232){$A$}
\rput[bl](5.94,-2.225232){$A$}
\pscircle[linecolor=black, linewidth=0.04, dimen=outer](13.52,0.9547681){0.38}
\rput[bl](13.18,-1.285232){$B$}
\psline[linecolor=black, linewidth=0.04](13.28,0.6747681)(12.06,-1.3252319)
\psline[linecolor=black, linewidth=0.04](13.74,0.6747681)(14.18,-0.1652319)
\rput[bl](14.92,0.1947681){$\Large{=}$}
\psline[linecolor=black, linewidth=0.04](15.66,0.8747681)(15.68,-1.3852319)
\rput[bl](11.76,-1.2652318){$A$}
\rput[bl](15.3,-1.285232){$A$}
\rput[bl](2.48,-0.20523189){$\tau$}
\rput[bl](6.12,1.8947681){$\tau$}
\pscircle[linecolor=black, linewidth=0.04, dimen=outer](6.2,1.9747682){0.38}
\pscircle[linecolor=black, linewidth=0.04, dimen=outer](9.9,1.8747681){0.38}
\psline[linecolor=black, linewidth=0.04](9.58,1.714768)(8.08,0.014768105)
\psline[linecolor=black, linewidth=0.04](10.22,1.7347682)(10.22,-2.2652318)
\psline[linecolor=black, linewidth=0.04](8.38,-0.5452319)(8.4,-2.2652318)
\rput[bl](8.0,-2.225232){$B$}
\pscircle[linecolor=black, linewidth=0.04, dimen=outer](14.26,-0.3052319){0.18}
\rput[bl](0.86,1.8747681){$\tau$}
\psline[linecolor=black, linewidth=0.04](1.26,1.7947681)(3.554898,-0.19405542)(4.72,-2.345232)
\psline[linecolor=black, linewidth=0.04](0.64,1.714768)(0.34,-2.3052318)
\psline[linecolor=black, linewidth=0.04](2.22,-0.2852319)(0.38,-1.4452319)
\psline[linecolor=black, linewidth=0.04](2.88,-0.2452319)(2.88,-2.345232)
\psline[linecolor=black, linewidth=0.04](2.46,-0.44523188)(1.28,-2.3052318)
\psline[linecolor=black, linewidth=0.04](1.04,1.5747681)(1.14,-0.12523189)
\psline[linecolor=black, linewidth=0.04](1.14,-0.105231896)(1.26,-0.6852319)
\psline[linecolor=black, linewidth=0.04](1.14,0.2747681)(1.76,-0.4052319)
\psline[linecolor=black, linewidth=0.04](1.9,-0.6252319)(2.1,-0.8652319)
\psline[linecolor=black, linewidth=0.04](2.22,-0.9852319)(2.88,-1.7452319)
\psline[linecolor=black, linewidth=0.04](1.34,-0.9452319)(1.64,-1.7652318)
\rput[bl](8.02,-0.4852319){$\gamma$}
\rput[bl](9.84,1.8147681){$\tau$}
\psline[linecolor=black, linewidth=0.04](9.82,1.5347681)(9.32,-2.2652318)
\psline[linecolor=black, linewidth=0.04](7.84,-0.5852319)(5.52,-2.225232)(5.54,-2.205232)
\psline[linecolor=black, linewidth=0.04](6.5,1.7947681)(7.98,1.1747681)(8.38,0.5747681)
\psline[linecolor=black, linewidth=0.04](7.98,1.1547681)(8.78,1.0147681)(8.78,1.0147681)
\psline[linecolor=black, linewidth=0.04](9.02,0.9547681)(9.58,0.8147681)
\psline[linecolor=black, linewidth=0.04](9.82,0.75476813)(10.2,0.5747681)
\psline[linecolor=black, linewidth=0.04](8.52,0.3747681)(9.48,-0.9052319)
\psline[linecolor=black, linewidth=0.04](6.2,1.6347681)(6.44,-0.045231894)(7.2,-0.8452319)
\psline[linecolor=black, linewidth=0.04](7.38,-1.0252318)(8.38,-1.7052319)(8.4,-1.7052319)
\psline[linecolor=black, linewidth=0.04](5.88,1.7947681)(6.22,-1.7452319)
\rput[bl](8.98,-2.2452319){$B$}
\rput[bl](4.2,-2.2452319){$B$}
\rput[bl](2.92,-2.225232){$B$}
\rput[bl](1.5,-2.205232){$B$}
\rput[bl](13.46,0.8947681){$\tau$}
\psline[linecolor=black, linewidth=0.04](13.5,0.6147681)(13.5,-1.3252319)
\pscircle[linecolor=black, linewidth=0.04, dimen=outer](15.66,1.0147681){0.18}
\psline[linecolor=black, linewidth=0.04](16.52,0.8747681)(16.54,-1.3852319)
\pscircle[linecolor=black, linewidth=0.04, dimen=outer](16.52,1.0147681){0.18}
\rput[bl](17.12,0.1747681){$\Large{=}$}
\pscircle[linecolor=black, linewidth=0.04, dimen=outer](18.5,0.9547681){0.38}
\rput[bl](14.02,0.23476811){$B$}
\psline[linecolor=black, linewidth=0.04](18.5,0.5947681)(18.5,-0.44523188)
\rput[bl](17.66,-1.3052319){$A$}
\pscircle[linecolor=black, linewidth=0.04, dimen=outer](18.52,-0.5852319){0.18}
\rput[bl](18.46,0.8747681){$\tau$}
\psline[linecolor=black, linewidth=0.04](18.82,0.7947681)(19.82,-1.4452319)
\rput[bl](16.62,-1.3052319){$B$}
\rput[bl](19.82,-1.3252319){$B$}
\rput[bl](18.56,-0.105231896){$B$}
\psline[linecolor=black, linewidth=0.04](18.22,0.7347681)(17.54,-1.4252319)
\end{pspicture}
}
\end{center}

\begin{proposition}\label{mopwrfromtwcoact}
Given a twisted coaction of a monoid $A$ on a bimonoid $B$ in a braided monoidal category $\CV$,
using the notation of Definition~\ref{twcoact}, a mixed opwreath around $A$ is defined by the
comonoid $B$ equipped with the morphisms
\begin{eqnarray*}
\zeta = (\eta_A\otimes 1_B)\bullet \gamma \dd B\otimes A \lra A\otimes B \\
\delta = (\eta_A\otimes \delta_B)\bullet \tau \dd B\lra A\otimes B\otimes B \\
\varepsilon = \eta_A\circ \varepsilon_B \dd B\lra A \phantom{real lot of stuff} 
\end{eqnarray*}
as required by Definition~\ref{defmopwr}.
\end{proposition}

Here are the string diagrams for these $\zeta$, $\delta$ and $\varepsilon$.

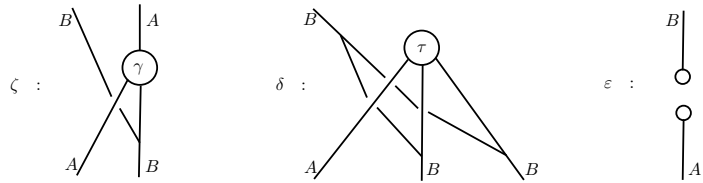
\begin{center}
\psscalebox{0.6 0.6} 
{
\begin{pspicture}(0,-1.9500768)(15.12,1.9500768)
\rput[bl](0.0,-0.009923172){$\Huge{\zeta \ :}$}
\rput[bl](5.82,0.050076827){$\Huge{\delta \ :}$}
\rput[bl](13.02,0.09007683){$\Huge{\varepsilon \ :}$}
\pscircle[linecolor=black, linewidth=0.04, dimen=outer](2.84,0.55007684){0.4}
\rput[bl](2.7,0.41007683){$\gamma$}
\rput[bl](8.92,0.8900768){$\tau$}
\psline[linecolor=black, linewidth=0.04](2.58,0.29007682)(1.46,-1.8299232)
\psline[linecolor=black, linewidth=0.04](2.84,0.93007684)(2.84,1.9500768)
\psline[linecolor=black, linewidth=0.04](2.86,0.15007682)(2.82,-1.8699231)
\psline[linecolor=black, linewidth=0.04](2.84,-1.1099231)(2.4,-0.34992316)(2.4,-0.34992316)
\psline[linecolor=black, linewidth=0.04](2.22,-0.08992317)(1.36,1.8700768)
\pscircle[linecolor=black, linewidth=0.04, dimen=outer](9.02,0.99007684){0.4}
\psline[linecolor=black, linewidth=0.04](8.72,0.73007685)(6.66,-1.9099232)
\psline[linecolor=black, linewidth=0.04](9.04,0.61007684)(9.02,-1.9499232)
\psline[linecolor=black, linewidth=0.04](9.32,0.7700768)(11.26,-1.9299232)
\psline[linecolor=black, linewidth=0.04](6.64,1.8500768)(8.22,0.3300768)
\psline[linecolor=black, linewidth=0.04](9.16,-0.42992318)(10.88,-1.3899232)
\psline[linecolor=black, linewidth=0.04](8.44,0.15007682)(8.94,-0.34992316)
\psline[linecolor=black, linewidth=0.04](8.04,-0.32992318)(9.02,-1.4099232)
\psline[linecolor=black, linewidth=0.04](7.82,-0.10992317)(7.24,1.2700769)
\psline[linecolor=black, linewidth=0.04](14.76,1.8100768)(14.74,0.49007684)
\psline[linecolor=black, linewidth=0.04](14.76,-0.58992314)(14.74,-1.9299232)
\pscircle[linecolor=black, linewidth=0.04, dimen=outer](14.74,0.35007682){0.18}
\pscircle[linecolor=black, linewidth=0.04, dimen=outer](14.76,-0.44992316){0.18}
\rput[bl](2.98,1.4700768){$A$}
\rput[bl](2.96,-1.7499232){$B$}
\rput[bl](1.2,-1.7099231){$A$}
\rput[bl](6.44,-1.7699232){$A$}
\rput[bl](14.86,-1.8299232){$A$}
\rput[bl](1.06,1.4700768){$B$}
\rput[bl](6.42,1.4900768){$B$}
\rput[bl](9.12,-1.8299232){$B$}
\rput[bl](11.28,-1.8299232){$B$}
\rput[bl](14.36,1.4700768){$B$}
\end{pspicture}
}

\end{center}

In proving Proposition~\ref{mopwrfromtwcoact}, a lemma will be useful.

\begin{lemma}\label{mopwrlem} The following equations hold:
\begin{itemize}
\item[(i)] $\delta_B\bullet \delta_B = \delta_B \circ \mu_B$
\item[(ii)] $(1\otimes \delta_B)\bullet \tau = (1_A\otimes \eta_B\otimes \eta_B) \bullet \delta$
\item[(iii)] $(\eta_A\otimes \delta_B)\bullet ((1\otimes \delta_B)\circ \gamma) = ((1\otimes \delta_B)\circ \zeta)$ 
\end{itemize}
\end{lemma}
\begin{proof}
Item (i) is a restatement of the bimonoid axiom for $B$ asserting that $\mu_B$ preserves comultiplication.
Item (ii) is immediate on drawing the string diagrams.
Item (iii) is immediate from the string diagrams and using the bimonoid axiom. 
\end{proof}

Here now are some clues on proving Proposition~\ref{mopwrfromtwcoact}.
There are seven conditions satisfied by the twisted coaction.
There are seven axioms to verify for the mixed opwreath.
For condition 1, we can express the fact that $\gamma$ preserves multiplication 
in the form $\gamma \circ \mu_A = \gamma \bullet \gamma$, then dot both sides on
the left with $\eta_A\otimes 1_B$.
Condition 2 follows by dotting on both sides by $\eta_A\otimes 1_B$ 
the equation expessing the fact that $\gamma$ preserves unit.
Condition 3 follows from counitality of $\gamma$ and the bimonoid condition that
$\mu_B$ preserves counit. Condition 4 is obtained by dotting both sides of the
$\tau$-coassociativity equation on the left by $\eta_A\otimes \delta_B$ and employing
Lemma~\ref{mopwrlem}. Condition 5 follows by dotting both sides of the
cocyclicity condition by $\eta_A\otimes \mu_{B 3}$, where $\mu_{B 3}$ is the ternary 
multiplication $\mu_B\circ (\mu_B\otimes 1_B) = \mu_B\circ (1_B\otimes \mu_B)$,
and employing Lemma~\ref{mopwrlem}. 
Unsurprisingly by now, conditions 6 and 7 follow from the two equations of normality
and that $\mu_B$ preserves counit.  

\begin{remark}
{\em If the 2-cocycle $\tau$ has the form $\eta_A\otimes \tau'$ for
some $\tau'\dd I\to B\otimes B$ then the mixed opwreath of 
Proposition~\ref{mopwrfromtwcoact} is a mixed opdistributive law.}
\end{remark}

There is a 2-categorical viewpoint on twisted coactions.
Recall (see \cite{BTC} or Chapter 15 of \cite{90}, for example)
that the category $\mathrm{Mon}\CV$ of monoids in the braided monoidal
category $\CV$ is a monoidal 2-category. If $f,g\dd M \to N$ are monoid morphisms
then a 2-cell $\xi\dd f\Rightarrow g\dd M\to N$ is a morphism $\xi\dd I\to N$ in $\CV$
satisfying the {\em naturality condition} 
$\xi\bullet f = g\bullet \xi$.  
The vertical composite of $\xi$ with a 2-cell $\zeta\dd g\Rightarrow h\dd M\to N$ is $\zeta \bullet \xi$.  
The horizontal composite of $\xi$ with a 2-cell $\xi'\dd f'\Rightarrow g'\dd N\to L$ is $\xi'\bullet (f'\circ \xi) = (g'\circ \xi) \bullet \xi'$.
The tensor product in $\mathrm{Mon}\CV$ is the tensor product of monoids that we have been dealing with already 
(it uses the braiding of $\CV$ yet is not itself a braided tensor product unless $\CV$ is symmetric). 

Now we can think of our bimonoid $B$ as a comonoid in the 2-category $\mathrm{Mon}\CV$.

\begin{proposition}
A twisted right coaction of the bimonoid $B$ on the monoid $A$ in $\CV$ is precisely a normal
lax right coaction of the comonoid $B$ on the object $A$ in the 2-category $\mathrm{Mon}\CV$.
\end{proposition} 
\begin{proof}
We need to see what is involved in a normal lax right coaction.
Indeed we have a morphism $\gamma \dd A\to A\otimes B$ in $\mathrm{Mon}\CV$, as required.
We have a 2-cell
\begin{equation}
\begin{aligned}
\xymatrix{
A \ar[d]_{\gamma}^(0.5){\phantom{AAAAAA}}="1" \ar[rr]^{\gamma}  && A\otimes B \ar[d]^{1\otimes \delta_B}_(0.5){\phantom{AAAAAA}}="2" \ar@{=>}"1";"2"^-{\tau}
\\
A\otimes B \ar[rr]_-{\gamma\otimes 1} && A\otimes B\otimes B 
}
\end{aligned}
\end{equation}
in $\mathrm{Mon}\CV$; the 2-cell condition is precisely $\tau$-coassociativity. 
A lax coaction also involves a 2-cell $\tau_0 \dd 1_A\Rightarrow (1_A\otimes \varepsilon) \circ \gamma$
however the normality condition is that this should be an identity; this precisely amounts to counitality. 
The axioms on $\tau$ for a lax coaction are precisely cocyclicity and normality.    
\end{proof}

We now remind the reader of the role that variants of the (algebraist's) simplicial category 
$\mathbf{\Delta}$ play as host to generic monoids, comonoids, actions and coactions (see \cite{Lawvere1969}). 
We write $\mathbf{\Delta}_{\bot,\top}$ for the strict monoidal category whose objects are
the strictly positive finite ordinals, whose morphisms are order and first-and-last-element preserving
functions; the tensor product is $m\oplus n = m+n-1$ thought of, for the purposes of the value
at morphisms, as identifying the last element of $m$ with the first element of $n$.
Similarly, $\mathbf{\Delta}_{\top}$ denotes the category whose objects are
the strictly positive finite ordinals, whose morphisms are order and last-element preserving
functions. There is a strict right action 
\begin{eqnarray}
\oplus \dd \mathbf{\Delta}_{\top}\otimes \mathbf{\Delta}_{\bot,\top} \lra \mathbf{\Delta}_{\top}
\end{eqnarray}
of $\mathbf{\Delta}_{\bot,\top}$ on $\mathbf{\Delta}_{\top}$ defined by $m\oplus n = m+n-1$ 
as before except that on morphisms, the left morphism in the operation need not preserve the first element,
so the result may not either. 

Here is a picture of some generating morphisms of $\mathbf{\Delta}_{\top}$.
\begin{eqnarray}\label{Deltop}
\UseTips{}\entrymodifiers={+<4mm>!C}\xymatrix{
\mathbf{\Delta}_{\top} \ :
& 1  \ar @<5pt>[r]^-{\partial_0}   
& 2  \ar @<15pt>[r]^-{\partial_0}  \ar @<-6pt>[r]^-{\partial_1} \ar @<5pt>[l]_-{\sigma_0}
& 3 \ar[l]<-6pt>_-{\sigma_0} \ar[l]<15pt>_-{\sigma_1}  \ar @<21pt>[r]^-{\partial_0}  \ar @<2pt>[r]^-{\partial_1}  \ar @<-18pt>[r]^-{\partial_2}
& \dots \ar[l]<-12pt>_-{\sigma_0} \ar[l]<7pt>_-{\sigma_1} \ar[l]<26pt>_-{\sigma_2} \,. }
\end{eqnarray}
 
The corresponding picture for $\mathbf{\Delta}_{\bot,\top}$ is obtained by deleting all the 
morphisms labelled $\partial_0$. 
There is a canonical inclusion $\mathbf{\Delta}_{\bot,\top}\rightarrowtail \mathbf{\Delta}_{\top}$
which respects the right actions by $\mathbf{\Delta}_{\bot,\top}$. 
The corresponding picture for $\mathbf{\Delta}$ is obtained by adjoining the object $0$ and 
morphisms $\partial_n\dd n \to (n+1)$. 
There is a canonical inclusion $\mathbf{\Delta}_{\top}\rightarrowtail \mathbf{\Delta}$.  
Moreover, $\mathbf{\Delta}^{\mathrm{op}} \cong \mathbf{\Delta}_{\bot,\top}$.

A comonoid $B=(B,\delta_B,\varepsilon_B)$ in a monoidal category $\CW$ defines a strong monoidal functor 
$\bar{B}\dd \mathbf{\Delta}_{\bot,\top} \lra \CW$ whose value at $n$ is $B^{\otimes (n-1)}$, whose
value at $\sigma_r\dd (n+1)\to n$ is 
\begin{eqnarray}
\sigma_r = 1_{B^{\otimes r}}\otimes \varepsilon_B \otimes 1_{B^{\otimes (n-r-1)}}\dd B^{\otimes n}\to B^{\otimes (n-1)} \ ,
\end{eqnarray}
and whose value at $\partial_r\dd n\to (n+1)$ is 
\begin{eqnarray}
\partial_r = 1_{B^{\otimes r}}\otimes \delta_B \otimes 1_{B^{\otimes (n-r-2)}}\dd B^{\otimes (n-1)}\to B^{\otimes n} \ .
\end{eqnarray}
 In fact this gives an equivalence of categories implying that, up to isomorphism, the comonoid $B$ can be recaptured from the strong monoidal functor.
 
 Suppose $\CW$ acts on a category $\CA$ via a functor $\star\dd \CA\times \CW\to \CA$. 
 The comonoid $B$ in $\CW$ defines a comonad $-\star B$ on $\CA$.
 We define a {\em right action} of $B$ on an object $A\in \CA$ to be the structure 
 $\gamma\dd A\to A\star B$ of an Eilenberg-Moore $(-\star B)$-coalgebra on $A$. 
 There is a functor $\bar{A}\dd \mathbf{\Delta}_{\top}\lra \CA$ whose value at the object
 $n$ is $A\star B^{\otimes (n-1)}$, whose value at $\partial_0\dd n\to (n+1)$ is 
 \begin{eqnarray*}
\gamma \otimes 1_{B^{\otimes (n-1)}} \dd A\star B^{\otimes (n-1)}\lra A\star B^{\otimes n} \ , 
\end{eqnarray*}
 and whose value at the other morphisms $\partial_{r+1}$ in \eqref{Deltop} is 
 $\partial_{r+1}=1_A\star \partial_r$ where $\partial_r$ comes from $\bar{B}$. 
Then $\bar{A}$ and $\bar{B}$ comprise an action morphism.
 \begin{eqnarray}
 \begin{aligned}
\xymatrix{
\mathbf{\Delta}_{\top}\times \mathbf{\Delta}_{\bot, \top} \ar[rr]^-{\oplus} \ar[d]_-{\bar{A}\times \bar{B}} && \mathbf{\Delta}_{\top} \ar[d]^-{\bar{A}} \\
\CA\times \CW \ar[rr]_-{\star} && \CA}
\end{aligned}
\end{eqnarray} 
Again, this is part of an equivalence of categories between action morphisms and pairs $(A,B)$.

This is all standard material, albeit maybe not explicitly in the above dual version.

Now suppose $\CA$ is a 2-category and the action $\star\dd \CA\times \CW\to \CA$
corresponds to a functor $\CW \to [\CA,\CA]$ into the 2-functor 2-category.
Suppose $A\in \CA$ has merely a morphism $\gamma \dd A\to A \star B$.
We can define $\bar{A}$ on objects and generating morphisms as before but it is not
quite a functor.

\begin{proposition}\label{Abar}
A normal lax $(-\star B)$-coalgebra structure on $\gamma \dd A\to A \star B$ 
amounts to a normal lax functor structure on $\bar{A}\dd \mathbf{\Delta}_{\top} \to \CA$ 
which has its constraints 
$\bar{A}(\zeta)\circ \bar{A}(\xi)\to \bar{A}(\zeta\circ \xi)$ 
identities unless neither $\xi$ nor $\zeta$ is in $\mathbf{\Delta}_{\bot,\top}$.
In particular, $\bar{A}$ restricts along the inclusion
$\mathbf{\Delta}_{\bot,\top} \rightarrowtail \mathbf{\Delta}_{\top}$ 
to a strict functor, that is, a simplicial object of $\CA$.  
\end{proposition}

\begin{proposition}\label{Ahat}
In the situation of Proposition~\ref{Abar}, suppose $A$ is pointed by a morphism $\eta_A\dd I\to A$ in $\CA$, 
the comonoid $B$ is pointed by a comonoid morphism $\eta_B\dd I\to B$ in $\CW$, and 
$\gamma \dd A\to A \star B$ respects the pointings, then
each lax functor $\bar{A}\dd \mathbf{\Delta}_{\top} \to \CA$  
extends along the inclusion $\mathbf{\Delta}_{\top} \rightarrowtail \mathbf{\Delta}$
to a lax functor $\hat{A}\dd \mathbf{\Delta} \to \CA$ by defining $\hat{A}(0)=I$, $\partial_0 = \eta_A \dd I\to A$, and
\begin{eqnarray*}
\partial_n \dd = 1_{A\otimes B^{\otimes(n-1)}}\otimes \eta_B \dd A\otimes B^{\otimes(n-1)}\to A\otimes B^{\otimes n}  \ .
\end{eqnarray*}     
\end{proposition}

In particular, for braided monoidal $\CV$ and $\CW=\CA=\mathrm{Mon}\CV$ 
(with the action on itself by its own tensor product), each twisted coaction of
a bimonoid $B$ on a monoid $A$ determines a slightly lax (augmented) cosimplicial monoid  
$\hat{A}$ in $\CV$ with $\hat{A}(n) = A\otimes B^{\otimes (n-1)}$. 
Our terminology that $\tau$ is a normalized 2-cocycle is justified by the formulas
\begin{eqnarray*}
(\partial_1\tau)\bullet (\partial_3\tau) = (\partial_2\tau)\bullet (\partial_0\tau) \ \text{ , } \ \sigma_1\tau = 1 = \sigma_0\tau \ .  
\end{eqnarray*} 
 
\section{Monoidality}\label{monoidality}

The basis of this section is the pioneering work of Day \cite{DayNoMM, DayPFC}.

Suppose $T=(T,\mu,\eta)$ is a monoidal monad on the monoidal category $\CA$. 
Then the Kleisli category $\CA_T$ is canonically monoidal: 
on objects, which are the same as for 
$\CA$, the tensor product is that of $\CA$; on homs it is equal to
\begin{eqnarray*}
\CA_T(X,Y)\times \CA_T(X',Y')= \CA(X,TY)\times \CA(X',TY')\stackrel{\otimes}\lra \CA(X\otimes X',TY\otimes TY') \\
\stackrel{\CA(1_{X\otimes X'},\phi_{Y,Y'})}\lra \CA(X\otimes X',T(Y\otimes Y')) = \CA_T(X\otimes X',Y\otimes Y') \ .
\end{eqnarray*}
The canonical functor $\CA\to \CA_T$ is strict monoidal.

\begin{definition}\label{monopmonopwr}
{\em A mixed opwreath $(G,\zeta,\delta,\varepsilon)$ around the monoidal monad $T$
on $\CA$ (see Definition~\ref{defmopwr}) is {\em opmonoidal} when the lifted comonad
$\bar{G}=(\bar{G},\bar{\delta},\bar{\varepsilon})$ on $\CA_T$ is equipped with opmonoidal
structure.}
\end{definition}
By the dual of the fact that Kleisli categories of monoidal monads are canonically monoidal,
for an opmonoidal opwreath, the category $\mathrm{mkl}(G,\zeta,T)$ is canonically monoidal since it is the Kleisli category for the opmonoidal comonoid $\bar{G}$ on $\CA_T$. 

Let us spell out the data and axioms involved in Definition~\ref{monopmonopwr}, and the
monoidal structure on $\mathrm{mkl}(G,\zeta,T)$.

The data are morphisms
\begin{eqnarray}\label{Gpsi}
\psi_{X,X'}\dd G(X\otimes X')\lra T(GX\otimes GX')
\end{eqnarray}
indexed by pairs of objects $X,X'\in \CA$.
There are six axioms.
\begin{eqnarray}\label{Gpsi1}
\begin{aligned}
\xymatrix{
G(X\otimes X') \ar[r]^-{\psi_{X,X'}} \ar[d]_-{G(f\otimes f')} & T(GX\otimes GX') \ar[rr]^-{ T(Gf\otimes Gf')}  && T(GTY\otimes GTY') \ar[d]^-{T(\zeta_Y\otimes \zeta_{Y'})} \\ 
G(TY\otimes TY') \ar[d]_-{G\phi_{Y,Y'}}  & & & T(TGY\otimes TGY') \ar[d]^-{T\phi_{GY,GY'}} \\
GT(Y\otimes Y')  \ar[d]_-{\zeta_{Y\otimes Y'}} & & & TT(GY\otimes GY')  \ar[d]^-{\mu_{GY\otimes GY'}}\\
TG(Y\otimes Y') \ar[r]_-{T\psi_{Y,Y'}} & TT(GY\otimes GY') \ar[rr]_-{ \mu_{GY\otimes GY'}}  && T(GY\otimes GY')   
}
\end{aligned}
\end{eqnarray}

\begin{eqnarray}\label{Gpsi2}
\begin{aligned}
\xymatrix{
G(X\otimes X'\otimes X'') \ar[rr]^-{\psi_{X\otimes X',X''}} \ar[d]_-{\psi_{X,X'\otimes X''}} & & T(G(X\otimes X')\otimes GX'') \ar[d]^-{ T(\psi_{X,X'}\otimes \eta_{GX''})}  \\
T(GX\otimes G(X'\otimes X'')) \ar[d]_-{T(\eta_{GX}\otimes \psi_{X',X''})} && T(T(GX\otimes GX')\otimes TGX'') \ar[d]^-{T\phi_{GX\otimes GX',GX''}} \\ 
T(TGX\otimes T(GX'\otimes GX'')) \ar[d]_-{T\phi_{GX,GX'\otimes GX''}}  &  & TT(GX\otimes GX'\otimes GX'') \ar[d]^-{\mu_{GX\otimes GX'\otimes GX''}} \\
TT(GX\otimes GX'\otimes GX'') \ar[rr]_-{\mu_{GX\otimes GX'\otimes GX''}} & & T(GX\otimes GX'\otimes GX'')   
}
\end{aligned}
\end{eqnarray}

\begin{eqnarray}\label{Gpsi3}
\begin{aligned}
\xymatrix{
& T(TGX\otimes TI) \ar[rd]^-{ T\phi_{GX,I}}  & \\
T(GX\otimes GI) \ar[ru]^-{T(\eta_{GX}\otimes \varepsilon_I)} & & TTGX \ar[d]^-{\mu_{GX}} \\
GX \ar[rr]_-{\eta_{GX}}  \ar[u]^-{\psi_{X,I}} & & TGX }
\end{aligned}
\end{eqnarray}

\begin{eqnarray}\label{Gpsi4}
\begin{aligned}
\xymatrix{
& T(TI \otimes TGX) \ar[rd]^-{ T\phi_{I,GX}}  & \\
T(GI\otimes GX) \ar[ru]^-{T(\varepsilon_I\otimes \eta_{GX})} & & TTGX \ar[d]^-{\mu_{GX}} \\
GX \ar[rr]_-{\eta_{GX}}  \ar[u]^-{\psi_{I,X}} & & TGX }
\end{aligned}
\end{eqnarray}

\begin{eqnarray}\label{Gpsi5}
\begin{aligned}
\xymatrix{
& TGG(X\otimes X') \ar[rd]^-{ TG\psi_{X,X'}} \\
G(X\otimes X') \ar[ru]^-{\delta_{X\otimes X'}} \ar[d]_-{\psi_{X,X'}} &   & TGT(GX\otimes GX') \ar[d]^-{T\zeta_{GX\otimes GX'}} \\ 
T(GX\otimes GX') \ar[d]_-{T(\delta_X\otimes \delta_X')}  & &  TTG(GX\otimes GX') \ar[d]^-{TT\psi_{GX,GX'}} \\
T(TGGX\otimes TGGX')  \ar[d]_-{T\phi_{GGX,GGX'}} & &  TTT(GGX\otimes GGX')  \ar[d]^-{T\mu_{GGX\otimes GGX'}}\\
TT(GGX\otimes GGX') \ar[rd]_-{\mu_{GGX\otimes GGX'}\phantom{aaa}} &  & TT(GGX\otimes GGX') \ar[ld]^-{\phantom{aaa}\mu_{GGX\otimes GGX'}} \\
&  T(GGX\otimes GGX')    
}
\end{aligned}
\end{eqnarray}

\begin{eqnarray}\label{Gpsi6}
\begin{aligned}
\xymatrix{
& T(TX \otimes TX') \ar[rd]^-{ T\phi_{X,X'}}  & \\
T(GX\otimes GX') \ar[ru]^-{T(\varepsilon_X\otimes \epsilon_{X'})} & & TT(X\otimes X') \ar[d]^-{\mu_{X\otimes X'}} \\
G(X\otimes X') \ar[rr]_-{\varepsilon_{X\otimes X'}}  \ar[u]^-{\psi_{X,X'}} & & T(X\otimes X') }
\end{aligned}
\end{eqnarray}

Diagram~\eqref{Gpsi1} expresses the naturality of $\psi$.
Diagrams~\eqref{Gpsi2},~\eqref{Gpsi3},~\eqref{Gpsi4} express the opmonoidality of $\bar{G}$ when equipped with $\psi$.
Diagrams~\eqref{Gpsi5},~\eqref{Gpsi6} express the opmonoidality of $\bar{\delta}$, $\bar{\varepsilon}$, respectively.
One of the other conditions is that the nullary piece $\psi_0$ of opmonoidal
structure on $\bar{G}$ must be $\varepsilon_I$. 
This means that the nullary conditions for $\bar{\delta}$, $\bar{\varepsilon}$
to be opmonoidal are automatically satisfied; the less trivial of these is the former,
which amounts to the Diagram~\eqref{GRedundant}, yet that follows using 3 and 6
for a mixed opwreath.

\begin{eqnarray}\label{GRedundant}
\begin{aligned}
\xymatrix{
TGTI \ar[r]^-{T\zeta_{I}} & TTGI \ar[r]^-{TT\varepsilon_I} & TTTI \ar[d]^-{\mu_{TI}} \\
TGGI \ar[u]^-{TG\varepsilon_{I}} & & TTI \ar[d]^-{\mu_I} \\
GI \ar[rr]_-{\varepsilon_{I}} \ar[u]^-{\delta_{I}} & & TI
 }
\end{aligned}
\end{eqnarray}

\bigskip
Now we come to the monoidal structure on $\mathrm{mkl}(G,\zeta,T)$.
The tensor product of two objects $X,X'$ is the tensor product $X\otimes X'$
of the objects as objects of $\CA$. 
The tensor product of morphisms $f\dd X\to Y$ and $f'\dd X'\to Y'$ in 
$\mathrm{mkl}(G,\zeta,T)$, which are morphisms 
$f\dd GX\to TY$ and $f'\dd GX'\to TY'$ in $\CA$, is the composite 
\begin{eqnarray*}
G(X\otimes X')\stackrel{\psi_{X,X'}}\lra T(GX\otimes GX')\stackrel{T(f\otimes f')}\lra T(TY\otimes TY') \\
\stackrel{T\phi_{Y.Y'}}\lra TT(Y\otimes Y')\stackrel{\mu_{Y\otimes Y'}}\lra T(Y\otimes Y').
\end{eqnarray*}
The unit of this monoidal structure is the unit $I$ of $\CA$.
The Eckmann-Hilton argument \cite{EckHil} yields:

\begin{corollary}
If $(G,\zeta,\delta,\varepsilon)$ is an opmonoidal mixed opwreath around 
the monoidal monad $T$ on $\CA$ then the monoid $\mathrm{mkl}(G,\zeta,T)(I,I)$
of endomorphisms of $I$ is commutative. 
\end{corollary} 

\section{Monoidal twisted coactions}

In this section we will show what structure on a twisted coaction leads to opmonoidality
of the generated mixed opwreath of Proposition~\ref{mopwrfromtwcoact}.
 
We work in a braided monoidal category $\CV$.

\begin{definition}\label{montwcoact}
{\em A twisted coaction $(\gamma,\tau)$ of a bimonoid $B$ on a monoid $A$ (see Definition~\ref{twcoact})
is {\em monoidal} when it is equipped with a morphism $\mathfrak{d}\dd B\to A\otimes B\otimes B$ which satisfies the five conditions \eqref{montwcoact1} to \eqref{montwcoact4}.}
\end{definition}
\begin{center}
\begin{equation}\label{montwcoact1}
\begin{aligned}
\psscalebox{0.55 0.55} 
{
\begin{pspicture}(0,-5.0366325)(18.119638,5.0366325)
\pscircle[linecolor=black, linewidth=0.04, dimen=outer](1.8373892,3.358535){0.43378124}
\pscircle[linecolor=black, linewidth=0.04, dimen=outer](3.964983,0.59059745){0.43378124}
\pscircle[linecolor=black, linewidth=0.04, dimen=outer](8.674608,1.7266912){0.43378124}
\psline[linecolor=black, linewidth=0.04](1.899358,4.9697223)(1.8787018,3.7303474)
\psline[linecolor=black, linewidth=0.04](8.633295,4.907754)(8.633295,2.11916)
\psline[linecolor=black, linewidth=0.04](8.922483,1.4168475)(8.943139,-4.903965)
\psline[linecolor=black, linewidth=0.04](1.5482018,3.0693474)(0.019639282,-4.862653)
\psline[linecolor=black, linewidth=0.04](3.6757956,0.34272245)(0.28817055,-3.6645901)
\psline[linecolor=black, linewidth=0.04](8.385421,1.4168475)(0.24685803,-3.7059026)(0.24685803,-3.7059026)
\psline[linecolor=black, linewidth=0.04](4.212858,0.30140996)(4.1922016,-0.97927755)
\psline[linecolor=black, linewidth=0.04](4.1715455,-1.433715)(4.1095767,-4.924621)
\psline[linecolor=black, linewidth=0.04](1.8580456,2.94541)(2.7462642,-0.48352754)
\psline[linecolor=black, linewidth=0.04](2.8288894,-0.81402755)(3.1800456,-1.6609337)
\psline[linecolor=black, linewidth=0.04](3.303983,-1.9501213)(4.130233,-3.6852462)
\psline[linecolor=black, linewidth=0.04](2.2092018,3.1313162)(6.2991395,0.32206622)(6.278483,0.34272245)
\psline[linecolor=black, linewidth=0.04](6.547014,0.13615996)(8.901827,-1.7848713)
\psline[linecolor=black, linewidth=0.04](3.964983,1.0243787)(3.964983,1.6647224)
\psline[linecolor=black, linewidth=0.04](3.964983,2.1811287)(3.9856393,4.9697223)(3.9856393,4.9697223)
\rput[bl](1.76,3.236535){$\mathfrak{d}$}
\rput[bl](3.8,0.44700372){$\gamma$}
\rput[bl](8.5,1.5863162){$\gamma$}
\rput[bl](11.179639,0.79255056){$\Huge{=}$}
\pscircle[linecolor=black, linewidth=0.04, dimen=outer](15.29964,2.8525505){0.4}
\pscircle[linecolor=black, linewidth=0.04, dimen=outer](15.45964,-1.3474494){0.4}
\psline[linecolor=black, linewidth=0.04](14.899639,5.032551)(15.279639,3.2125506)
\psline[linecolor=black, linewidth=0.04](16.539639,4.9925504)(15.139639,4.0125504)
\psline[linecolor=black, linewidth=0.04](15.139639,-1.5474495)(13.479639,-4.927449)
\psline[linecolor=black, linewidth=0.04](15.439639,-1.7274494)(15.439639,-5.0074496)
\psline[linecolor=black, linewidth=0.04](15.75964,-1.5874494)(17.539639,-5.0274496)
\psline[linecolor=black, linewidth=0.04](15.539639,2.5725505)(15.45964,-0.9674494)
\psline[linecolor=black, linewidth=0.04](15.019639,2.5925505)(14.019639,-3.7874494)
\psline[linecolor=black, linewidth=0.04](13.139639,5.0125504)(14.739639,1.7125506)
\psline[linecolor=black, linewidth=0.04](14.91964,1.4725506)(15.479639,0.23255058)
\rput[bl](15.2,2.7063162){$\gamma$}
\rput[bl](15.4,-1.5){$\mathfrak{d}$}
\rput[bl](3.4196393,4.6525507){$A$}
\rput[bl](12.699639,4.6325507){$B$}
\rput[bl](8.039639,4.5925508){$A$}
\rput[bl](0.13963929,-4.7874494){$A$}
\rput[bl](13.619639,-4.8674493){$A$}
\rput[bl](14.33964,4.6525507){$A$}
\rput[bl](16.41964,4.6325507){$A$}
\rput[bl](17.47964,-4.8874493){$B$}
\rput[bl](15.45964,-4.8674493){$B$}
\rput[bl](8.439639,-4.8474493){$B$}
\rput[bl](4.1196394,-4.8274493){$B$}
\rput[bl](1.2996392,4.6525507){$B$}
\end{pspicture}
}
\end{aligned}
\end{equation}
\end{center}
\begin{center}
\begin{equation}\label{montwcoact2}
\begin{aligned}
\psscalebox{0.55 0.55} 
{
\begin{pspicture}(0,-4.1182923)(14.499621,4.1182923)
\pscircle[linecolor=black, linewidth=0.04, dimen=outer](1.5573716,2.5067673){0.43378124}
\psline[linecolor=black, linewidth=0.04](1.6193404,4.1179547)(1.5986842,2.8785796)
\psline[linecolor=black, linewidth=0.04](1.2681842,2.2175798)(0.01962166,-4.1144204)
\psline[linecolor=black, linewidth=0.04](1.578028,2.0936422)(1.9862467,0.52470475)
\psline[linecolor=black, linewidth=0.04](2.0288718,-0.24579528)(2.1000278,-4.0727015)
\rput[bl](1.4343404,2.3647673){$\mathfrak{d}$}
\rput[bl](0.13962166,-3.919217){$A$}
\rput[bl](1.0196216,3.800783){$B$}
\pscircle[linecolor=black, linewidth=0.04, dimen=outer](1.9973717,0.14676723){0.43378124}
\rput[bl](1.8743404,0.004767227){$\mathfrak{d}$}
\psline[linecolor=black, linewidth=0.04](1.6996217,-0.13921715)(0.45962167,-1.9392171)(0.47962165,-1.9592172)
\psline[linecolor=black, linewidth=0.04](2.2796216,-0.15921715)(4.1596217,-4.039217)(4.1596217,-4.039217)
\psline[linecolor=black, linewidth=0.04](1.8796216,2.280783)(5.559622,-4.059217)
\rput[bl](5.519622,-3.939217){$B$}
\rput[bl](4.1196218,-3.939217){$B$}
\rput[bl](2.1196218,-3.939217){$B$}
\rput[bl](6.1196218,0.68078285){$\Huge{=}$}
\pscircle[linecolor=black, linewidth=0.04, dimen=outer](10.037372,2.4667673){0.43378124}
\psline[linecolor=black, linewidth=0.04](10.03934,4.0979548)(10.018684,2.8585796)
\rput[bl](9.95434,2.324767){$\mathfrak{d}$}
\rput[bl](10.079621,3.800783){$B$}
\pscircle[linecolor=black, linewidth=0.04, dimen=outer](12.1973715,0.10676723){0.43378124}
\psline[linecolor=black, linewidth=0.04](10.399341,2.2779548)(12.178684,0.49857974)
\rput[bl](12.114341,-0.03523277){$\mathfrak{d}$}
\psline[linecolor=black, linewidth=0.04](9.679622,2.260783)(8.179622,-3.919217)(8.199621,-3.919217)
\psline[linecolor=black, linewidth=0.04](11.819622,-0.07921715)(8.619621,-1.9592172)
\psline[linecolor=black, linewidth=0.04](10.039621,2.080783)(9.999621,-0.87921715)
\psline[linecolor=black, linewidth=0.04](9.999621,-1.2992171)(9.979622,-3.979217)
\psline[linecolor=black, linewidth=0.04](12.159621,-0.27921715)(12.199621,-3.939217)(12.179622,-3.919217)
\psline[linecolor=black, linewidth=0.04](12.519622,-0.11921715)(13.819622,-3.959217)
\rput[bl](8.259622,-3.8992171){$A$}
\rput[bl](10.019622,-3.919217){$B$}
\rput[bl](12.259622,-3.939217){$B$}
\rput[bl](13.859622,-3.939217){$B$}
\end{pspicture}
}
\end{aligned}
\end{equation}
\end{center}
\begin{center}
\begin{equation}\label{montwcoact3}
\begin{aligned}
\psscalebox{0.55 0.55} 
{
\begin{pspicture}(0,-2.6667235)(14.2759,2.6667235)
\pscircle[linecolor=black, linewidth=0.04, dimen=outer](1.2370816,1.0551984){0.43378124}
\psline[linecolor=black, linewidth=0.04](1.2990505,2.666386)(1.2783942,1.427011)
\rput[bl](1.1140504,0.9131985){$\mathfrak{d}$}
\rput[bl](0.119331665,-2.590786){$A$}
\rput[bl](0.69933164,2.349214){$B$}
\psline[linecolor=black, linewidth=0.04](1.5593317,0.8292141)(2.1593316,0.009214096)
\rput[bl](7.4393315,-2.5107858){$B$}
\rput[bl](1.4393317,-2.590786){$B$}
\rput[bl](3.4793317,0.2892141){$\Huge{=}$}
\pscircle[linecolor=black, linewidth=0.04, dimen=outer](11.977081,1.0151985){0.43378124}
\psline[linecolor=black, linewidth=0.04](11.979051,2.646386)(11.958394,1.4070109)
\rput[bl](11.894051,0.87319845){$\mathfrak{d}$}
\rput[bl](12.019332,2.349214){$B$}
\psline[linecolor=black, linewidth=0.04](12.33905,0.826386)(14.258394,-2.652989)
\psline[linecolor=black, linewidth=0.04](11.979332,0.6292141)(11.939332,-0.3507859)
\rput[bl](5.8193316,-2.5107858){$A$}
\rput[bl](13.479332,-2.4907858){$B$}
\rput[bl](2.0393317,0.1892141){$B$}
\rput[bl](12.019332,-0.15078591){$B$}
\psline[linecolor=black, linewidth=0.04](0.91933167,0.8092141)(0.019331666,-2.590786)
\pscircle[linecolor=black, linewidth=0.04, dimen=outer](2.2493317,-0.1407859){0.17}
\rput[bl](8.359332,0.3092141){$\Huge{=}$}
\psline[linecolor=black, linewidth=0.04](1.3193316,0.6492141)(1.3393316,-2.610786)
\pscircle[linecolor=black, linewidth=0.04, dimen=outer](11.949331,-0.4807859){0.17}
\pscircle[linecolor=black, linewidth=0.04, dimen=outer](6.389332,0.8192141){0.17}
\psline[linecolor=black, linewidth=0.04](6.3993316,0.6492141)(6.4193316,-2.570786)
\psline[linecolor=black, linewidth=0.04](7.2993317,2.649214)(7.3193316,-2.590786)
\rput[bl](7.3793316,2.349214){$B$}
\psline[linecolor=black, linewidth=0.04](11.639332,0.7892141)(10.759332,-2.570786)
\rput[bl](10.939332,-2.5107858){$A$}
\end{pspicture}
}
\end{aligned}
\end{equation}
\end{center}
\begin{center}
\begin{equation}\label{montwcoact4}
\begin{aligned}
\psscalebox{0.55 0.55} 
{
\begin{pspicture}(0,-3.6709888)(18.359861,3.6709888)
\pscircle[linecolor=black, linewidth=0.04, dimen=outer](0.97986066,2.1649487){0.36}
\psline[linecolor=black, linewidth=0.04](0.67986065,1.9849486)(0.01986065,-3.5750513)
\pscircle[linecolor=black, linewidth=0.04, dimen=outer](4.5198607,-0.4950513){0.36}
\pscircle[linecolor=black, linewidth=0.04, dimen=outer](8.179861,-0.4950513){0.36}
\psline[linecolor=black, linewidth=0.04](4.1998606,-0.5950513)(0.15986066,-2.2950513)
\psline[linecolor=black, linewidth=0.04](7.879861,-0.67505133)(0.19986065,-2.2950513)
\psline[linecolor=black, linewidth=0.04](8.199861,-0.8350513)(8.199861,-3.6350513)
\psline[linecolor=black, linewidth=0.04](8.459861,-0.6950513)(9.499861,-3.6150513)
\psline[linecolor=black, linewidth=0.04](1.2798606,2.0049486)(6.7798605,-0.7750513)
\psline[linecolor=black, linewidth=0.04](7.0598607,-0.9550513)(8.079861,-1.2350513)
\psline[linecolor=black, linewidth=0.04](8.299861,-1.2750514)(8.699861,-1.3950514)
\psline[linecolor=black, linewidth=0.04](6.3598604,-1.1350513)(8.219861,-2.3350513)
\psline[linecolor=black, linewidth=0.04](1.0398606,1.8649487)(2.4798605,-0.4950513)
\psline[linecolor=black, linewidth=0.04](4.499861,-0.8150513)(4.499861,-1.2750514)(4.5198607,-1.2550513)
\psline[linecolor=black, linewidth=0.04](4.499861,-1.5350513)(4.439861,-3.4950514)
\psline[linecolor=black, linewidth=0.04](4.7798605,-0.67505133)(5.0198607,-1.1750513)
\psline[linecolor=black, linewidth=0.04](5.1198606,-1.3350513)(6.0398607,-3.5550513)
\psline[linecolor=black, linewidth=0.04](2.4798605,-0.4950513)(3.2398605,-0.8150513)
\psline[linecolor=black, linewidth=0.04](3.5798607,-0.9550513)(4.379861,-1.0350513)
\psline[linecolor=black, linewidth=0.04](4.5998607,-1.0350513)(4.999861,-1.0750513)(4.999861,-1.0750513)
\psline[linecolor=black, linewidth=0.04](2.4998608,-0.5150513)(2.7798607,-1.0150514)(2.7798607,-1.0350513)(2.7998607,-1.0150514)
\psline[linecolor=black, linewidth=0.04](2.9398606,-1.2550513)(3.3998606,-1.4750513)
\psline[linecolor=black, linewidth=0.04](3.6798606,-1.6750513)(4.479861,-2.0550513)
\psline[linecolor=black, linewidth=0.04](1.0398606,3.6449487)(1.0398606,2.4849486)(1.0598607,2.4849486)
\psline[linecolor=black, linewidth=0.04](6.1398606,-0.8950513)(5.0398607,0.124948695)
\rput[bl](9.659861,0.7249487){$\Huge{=}$}
\rput[bl](0.8998606,2.0449486){$\mathfrak{d}$}
\rput[bl](8.119861,-0.6150513){$\tau$}
\rput[bl](4.399861,-0.6350513){$\tau$}
\pscircle[linecolor=black, linewidth=0.04, dimen=outer](13.839861,2.0249486){0.36}
\rput[bl](13.759861,1.8649487){$\tau$}
\psline[linecolor=black, linewidth=0.04](13.55986,1.8249487)(11.739861,-3.6550512)
\pscircle[linecolor=black, linewidth=0.04, dimen=outer](16.73986,0.6849487){0.36}
\pscircle[linecolor=black, linewidth=0.04, dimen=outer](15.699861,-0.19505131){0.36}
\pscircle[linecolor=black, linewidth=0.04, dimen=outer](14.299861,-1.9150513){0.36}
\psline[linecolor=black, linewidth=0.04](14.159861,1.8449486)(16.759861,1.0049487)(16.71986,1.0249487)
\psline[linecolor=black, linewidth=0.04](15.39986,-0.3550513)(12.039861,-2.7750514)(12.499861,-1.3550513)
\psline[linecolor=black, linewidth=0.04](14.01986,-2.0950513)(12.01986,-2.7750514)
\psline[linecolor=black, linewidth=0.04](14.31986,-2.2350514)(14.2798605,-3.6350513)
\psline[linecolor=black, linewidth=0.04](14.619861,-2.0950513)(15.679861,-3.6350513)
\psline[linecolor=black, linewidth=0.04](15.959861,-0.4150513)(14.379861,-1.5950513)(14.39986,-1.6150513)
\psline[linecolor=black, linewidth=0.04](13.879861,1.6849487)(14.679861,-0.6950513)
\psline[linecolor=black, linewidth=0.04](14.7798605,-0.8950513)(14.97986,-1.1750513)(14.999861,-1.1750513)
\psline[linecolor=black, linewidth=0.04](16.47986,0.46494868)(15.59986,0.14494869)(15.619861,0.14494869)
\psline[linecolor=black, linewidth=0.04](16.759861,0.3649487)(15.159861,-2.6750512)
\psline[linecolor=black, linewidth=0.04](15.039861,-2.8350513)(14.839861,-3.5950513)
\psline[linecolor=black, linewidth=0.04](16.97986,0.46494868)(17.439861,-3.5950513)
\psline[linecolor=black, linewidth=0.04](12.09986,3.6649487)(12.759861,1.5849487)
\psline[linecolor=black, linewidth=0.04](12.739861,1.6449487)(13.299861,1.4649487)
\psline[linecolor=black, linewidth=0.04](13.499861,1.4049487)(13.85986,1.3249487)
\psline[linecolor=black, linewidth=0.04](12.759861,1.6049486)(13.199861,1.1249487)
\psline[linecolor=black, linewidth=0.04](13.419861,0.9849487)(14.379861,0.24494869)(14.35986,0.2649487)
\psline[linecolor=black, linewidth=0.04](14.119861,1.3049487)(16.33986,1.1249487)
\rput[bl](16.679861,0.5449487){$\mathfrak{d}$}
\rput[bl](14.259861,-2.0550513){$\mathfrak{d}$}
\rput[bl](15.59986,-0.3350513){$\gamma$}
\rput[bl](14.919861,-3.5750513){$B$}
\rput[bl](13.799861,-3.5750513){$B$}
\rput[bl](9.499861,-3.5550513){$B$}
\rput[bl](8.259861,-3.5550513){$B$}
\rput[bl](6.0798607,-3.5350513){$B$}
\rput[bl](4.5398607,-3.5350513){$B$}
\rput[bl](1.1398606,3.3049488){$B$}
\rput[bl](12.31986,3.2649486){$B$}
\rput[bl](15.659861,-3.5550513){$B$}
\rput[bl](16.89986,-3.5150514){$B$}
\rput[bl](11.93986,-3.5950513){$A$}
\rput[bl](0.11986065,-3.5550513){$A$}
\end{pspicture}
}
\end{aligned}
\end{equation}
\end{center}

\begin{proposition} Let $\CV$ be a braided monoidal category.
Given a monoidal twisted coaction $(\gamma,\tau)$ of a bimonoid $B$ on a 
commutative monoid $A$ (see Definition~\ref{twcoact}), the mixed opwreath
described in Proposition~\ref{mopwrfromtwcoact}, equipped with the morphisms
\begin{eqnarray*}
\psi_{X,X'}= (1_A\otimes 1_B\otimes c^{-1}_{X,B}\otimes 1_{X'})\circ (\mathfrak{d}\otimes 1_X\otimes 1_{X'}) \ ,
\end{eqnarray*}
is opmonoidal.
\end{proposition}
\begin{proof}
Since $A$ is commutative, the arising monad $T=A\otimes -$ is monoidal with
$\phi_0=\eta\otimes 1_I$ and $\phi_{X,X'}= \mu \circ (1_A\otimes c_{X,A}\otimes 1_{X'})$.
Here are the string diagrams for $\psi$ and $\phi$.
\begin{center}
\begin{equation}\label{psiphi}
\begin{aligned}
\psscalebox{0.55 0.55} 
{
\begin{pspicture}(0,-2.7801228)(18.44,2.7801228)
\pscircle[linecolor=black, linewidth=0.04, dimen=outer](4.34,0.90012264){0.4}
\rput[bl](0.0,0.24012265){$\Huge{\psi_{X,X'} \ \ =}$}
\psline[linecolor=black, linewidth=0.04](4.36,2.7801228)(4.36,1.2601227)
\psline[linecolor=black, linewidth=0.04](4.06,0.64012265)(2.94,-2.5198774)(2.96,-2.5598774)
\psline[linecolor=black, linewidth=0.04](4.36,0.5401226)(4.38,-2.5198774)
\psline[linecolor=black, linewidth=0.04](4.6,0.6201227)(6.02,-2.5998774)
\psline[linecolor=black, linewidth=0.04](6.3,2.7601225)(5.4,-0.8798773)
\psline[linecolor=black, linewidth=0.04](5.28,-1.1998774)(5.08,-2.5398774)
\psline[linecolor=black, linewidth=0.04](7.08,2.7201226)(7.14,-2.6198773)
\rput[bl](4.26,0.78012264){$\mathfrak{d}$}
\psline[linecolor=black, linewidth=0.04](18.32,2.5601227)(18.38,-2.7798774)
\rput[bl](10.4,0.24012265){$\Huge{\phi_{X,X'} \ \ =}$}
\psline[linecolor=black, linewidth=0.04](13.96,2.5201225)(14.0,-2.7398775)
\psline[linecolor=black, linewidth=0.04](17.02,2.5401227)(14.02,-0.85987735)
\psline[linecolor=black, linewidth=0.04](15.14,2.5201225)(15.6,1.2201227)
\psline[linecolor=black, linewidth=0.04](15.76,0.8001226)(16.28,-2.6798773)
\rput[bl](3.82,2.4201226){$B$}
\rput[bl](2.38,-2.4598773){$A$}
\rput[bl](14.64,2.2201226){$X$}
\rput[bl](17.7,2.1801226){$X'$}
\rput[bl](6.1,-2.4798775){$B$}
\rput[bl](3.9,-2.4798775){$B$}
\rput[bl](5.68,2.4001226){$X$}
\rput[bl](7.14,2.4001226){$X'$}
\rput[bl](16.22,2.1801226){$A$}
\rput[bl](13.36,2.2201226){$A$}
\end{pspicture}
}
\end{aligned}
\end{equation}
\end{center}
With these data, and that of Proposition~\ref{mopwrfromtwcoact}, draw the string diagrams
for Diagrams~\eqref{Gpsi1} to \eqref{Gpsi6}. 
The remarkable fact is that the variables $f$ and $f'$ can be moved out of the top
of Diagrams~\eqref{Gpsi1}, while $X$ and $X'$ can be disconnected
from all the diagrams. 
Diagram~\eqref{Gpsi6} follows from Diagram~\eqref{montwcoact3}.
Then Diagrams~\eqref{montwcoact1} to \eqref{montwcoact4} are what remains. 
So indeed we obtain an opmonoidal mixed opwreath.
\end{proof}

\appendix

\end{document}